\documentclass{amsart}
\usepackage{amsmath}
\usepackage{amssymb,amsthm}
\usepackage[all]{xy}
\usepackage{graphicx}
\usepackage{tikz-cd}
\usepackage{adjustbox}
\usepackage{enumitem}
\usepackage[new]{old-arrows}
\usepackage{mathrsfs}
\usepackage{stmaryrd}
\usepackage[active]{srcltx}
\usepackage{euscript}

\usepackage{mathabx,epsfig}

\newcommand\itemb{\item[\textbf{(10.4.2):}]}
\newcommand\itemc{\item[\textbf{(10.4.3):}]}
\newcommand\itemd{\item[\textbf{(10.4.4):}]}
\newcommand\iteme{\item[\textbf{(10.4.5):}]}

\RequirePackage{ifpdf}
\ifpdf
  \IfFileExists{pdfsync.sty}{\RequirePackage{pdfsync}}{}
  \RequirePackage[pdftex, 
   colorlinks = true,
   pagebackref,
   bookmarksopen=true]{hyperref}
\else
  \RequirePackage[hypertex]{hyperref}
\fi

\usepackage{tikz}
\tikzset{
	MyPersp/.style={scale=2,x={(0.8cm,0cm)},y={(0cm,0.25cm)},
    z={(0cm,1cm)}},
	MyPoints/.style={fill=white,draw=black,thick}
		}

\usetikzlibrary{decorations.pathmorphing,arrows,calc}

 \definecolor{darkgreen}{HTML}{336633}
 \definecolor{darkred}{HTML}{993333}

\RequirePackage{color}
\definecolor{myred}{rgb}{0.75,0,0}
\definecolor{mygreen}{rgb}{0,0.5,0}
\definecolor{myblue}{rgb}{0,0,0.65}

\renewcommand{\L}{\mathbb{L}}

\newcommand{\K}{\mathbb{K}}
\renewcommand{\O}{\mathbb{O}}
\newcommand{\bk}{\Bbbk}

\newcommand{\Rl}{R_{\mathbb{K}}}

\DeclareMathOperator{\Spec}{Spec}
\newcommand{\pt}{\mathrm{pt}}

\newcommand{\Tilt}{\mathsf{Tilt}}

\newcommand{\Parity}{\mathsf{Parity}}

\newcommand{\Gr}{{\EuScript Gr}}
\newcommand{\Fl}{{\EuScript Fl}}

\newcommand{\Perv}{\mathsf{Perv}}

\newcommand{\IC}{\mathsf{IC}}

\newcommand{\End}{\mathrm{End}}

\makeatletter
\def\lotimes{\@ifnextchar_{\@lotimessub}{\@lotimesnosub}}
\def\@lotimessub_#1{\mathchoice{\mathbin{\mathop{\otimes}^L}_{#1}}%
  {\otimes^L_{#1}}{\otimes^L_{#1}}{\otimes^L_{#1}}}
\def\@lotimesnosub{\mathbin{\mathop{\otimes}^L}}
\makeatother

\newcommand{\For}{\mathsf{For}}

\newcommand{\id}{\mathrm{id}}

\DeclareMathOperator{\Hom}{Hom}

\DeclareMathOperator{\Res}{Res}

\DeclareMathOperator{\supp}{supp}

\newcommand{\pr}{\mathsf{pr}}

\newcommand{\excise}[1]{}

\newcommand{\Av}{\mathsf{Av}}

\theoremstyle{definition}

\newtheorem*{conj*}{Conjecture}
\newtheorem*{thm*}{Theorem}
\newtheorem*{cor*}{Corollary}
\numberwithin{equation}{section}
\newtheorem{thm}{Theorem}[section]
\newtheorem{lem}[thm]{Lemma}
\newtheorem{prop}[thm]{Proposition}
\newtheorem{cor}[thm]{Corollary}

\theoremstyle{definition}
\newtheorem{defn}[thm]{Definition}

\theoremstyle{remark}
\newtheorem{rmk}[thm]{Remark}

\newtheorem*{rmk*}{Remark}

\DeclareMathOperator{\Rep}{Rep}

\newcommand{\BSvar}{\mathsf{BS}}

\newcommand{\IW}{\text{IW}}

\newcommand{\Gm}{\mathbb{G}_{\mathrm{m}}}

\newcommand{\Smith}{\mathsf{Sm}}
\newcommand{\parity}{\mathrm{par}}

\newcommand{\Satake}{\mathsf{Sat}}

\author{Joshua Ciappara}
\date{June 2021}
\title{Hecke category actions via Smith--Treumann theory}

\begin{document}
\begin{abstract}
    Let $\textbf{G}$ be a simply connected semisimple algebraic group over a field of characteristic greater than the Coxeter number. We construct a monoidal action of the diagrammatic Hecke category on the principal block $\Rep_0(\textbf{G})$ of $\Rep(\textbf{G})$ by wall-crossing functors. This action was conjectured to exist by Riche--Williamson \cite{rw}. Our method uses constructible sheaves and relies on Smith--Treumann theory as applied in \cite{st}.
\end{abstract}
\maketitle

\section{Introduction}
A motivating problem in the study of reductive algebraic groups over a field $\bk$ of positive characteristic $\ell > 0$ is the determination of characters for important classes of modules. First conjectured by G. Lusztig, there exists a character formula for irreducible representations in terms of Kazhdan--Lusztig polynomials. This formula is known to be true for almost all $\ell$, but was shown in \cite{wil16c} not to hold under the original hypothesis $\ell > h$.

In response to this and other questions, Riche--Williamson \cite{rw} conjectured new formulas for simple and indecomposable tilting modules, applying to any $\ell > h$ and (after variation) perhaps all $\ell$. These formulas replace Kazhdan--Lusztig polynomials with \textit{$p$-Kazhdan--Lusztig polynomials}, which are suggested to be better suited to modular representation theory. The new conjectures are derived in \cite{rw} as a consequence of a more categorical proposition: the existence of an action of the diagrammatic Hecke category $\mathscr{H}$ (defined in \cite{ew}) on the principal block $\text{Rep}_0(\textbf{G})$ by wall-crossing functors, categorifying the action of the affine Weyl group on its antispherical module. Using methods from the theory of 2-Kac--Moody actions, Riche--Williamson proved their categorical conjecture for $\text{GL}_n$, but the general statement has remained open until recently. After the tilting and irreducible character formulas were established by other means \cite{amrw19} for $\ell > 2h -2$, the following year saw two major developments:
\begin{enumerate}
    \item The application of Smith--Treumann theory (as developed by Treumann \cite{tr} and Leslie--Lonergen \cite{ll}) by Riche--Williamson \cite{st}, providing a novel geometric proof of the linkage principle and establishing the aforementioned character formulas in all characteristics.
    \item The resolution of the categorical conjecture of Riche--Williamson in full generality by Bezrukavnikov--Riche \cite{bezri}. Their approach is essentially coherent, and makes use of localization theorems in characteristic $p$, as well as a new bimodule-theoretic realisation of $\mathscr{H}$ found by Abe \cite{abe1, abe2}.
\end{enumerate}
Our objective is to provide an alternative proof of Riche--Williamson's conjecture, using the machinery of constructible sheaves and Smith--Treumann theory. Assume now that $G$ is a semisimple algebraic group of adjoint type over an algebraically closed field $\mathbb{F}$ of characteristic $p \ne \ell$, with 
$$\textbf{G} = \Spec(\bk) \times_{\Spec(\mathbb{Z})} G_{\mathbb{Z}}^\vee,$$
where $G_{\mathbb{Z}}^\vee$ is the unique split reductive group scheme over $\mathbb{Z}$ whose base change to $\mathbb{C}$ has root datum dual to that of $G$. 

The starting point of our approach is a realisation of the Hecke category via parity sheaves on the affine flag variety $\Fl$ of $G$, first proved in \cite{rw} but modified to incorporate loop rotation $\mathbb{G}_m$-equivariance. This realisation induces a graded right $\mathscr{H}$-module equivalence between the antispherical quotient of $\mathscr{H}$ and $\Parity_{\IW}(\Fl,\bk)$, the category of Iwahori--Whittaker parity complexes on $\Fl$. Through an understanding of the morphism spaces between parity objects, as provided by \cite[\textsection 7]{st}, we show that the graded action of $\mathscr{H}$ descends further to a \textit{Smith quotient} $\Smith_{\IW}^\parity(\Fl,\bk)$ of $\Parity_{\IW}(\Fl,\bk)$. Up to graded shift, the indecomposable object $B_s \in \mathscr{H}$ acts on $\Smith_{\IW}^\parity(\Fl,\bk)$ by the functor induced by the composite $(q^s)^* (q^s)_*$, where $s$ is an affine simple reflection and
$$q^s: \Fl \to \Fl^s$$
is the natural morphism between certain partial affine flag varieties. 
We now use three ingredients to transfer the Hecke category action to $\Rep_0(G)$. Let $\Gr$ denote the affine Grassmannian of $G$, and let $\Gr^\varpi$ denote the fixed points of $\Gr$ under loop rotation by the $\ell$-th roots of unity $\varpi \le \mathbb{G}_m$. The first ingredient is the miraculous decomposition
\begin{equation} \label{decomp}
\Gr^\varpi = \bigsqcup_\nu \Gr_{(\nu)},
\end{equation}
where the components $\Gr_{(\nu)} = \Fl_\ell^\nu$ are (thin) partial affine flag varieties. The second ingredient is the main theorem of \cite{st}, which shows that Smith restriction from Iwahori--Whittaker perverse sheaves on $\Gr$ to the Smith category of $(\Gr)^\varpi$ is fully faithful on the tilting subcategory, with a well-understood essential image:
\begin{equation} \label{secing}
\begin{tikzcd}
\Perv_{\IW}(\Gr,\bk) \arrow{r}{} & \Smith_{\IW}(\Gr^\varpi,\bk) \\
\Tilt_{\IW}(\Gr,\bk) \arrow{r}{\cong} \arrow[hookrightarrow]{u} \arrow{r}{\cong} & \Smith_{\IW}^\natural(\Gr^\varpi,\bk). \arrow[hookrightarrow]{u}
\end{tikzcd}
\end{equation}
The categories in the lower row admit a decomposition into ``blocks'', preserved by the equivalence and ultimately tied to the linkage principle for $\textbf{G}$; our choices determine \textit{principal blocks} $\Tilt_{\IW}^0$ and $\Smith_\IW^0$ on the left and right sides, respectively. Our final ingredient is a version of the geometric Satake equivalence due to \cite{bgmrr}:
\begin{equation} \label{thirding}
    \Rep(\textbf{G}) \xrightarrow{\cong} \Perv_{\IW}(\Gr,\bk).
\end{equation}
Through the identification of $\Fl$ with a component of $\Gr^\varpi$ in \eqref{decomp}, $\Smith_\IW^0$ inherits a right $\mathscr{H}$-module action from $\Smith_\IW^\parity(\Fl,\bk)$. Using \eqref{secing} and \eqref{thirding}, we are able to transfer the action first to $\Tilt_{\IW}^0$ and hence to $\Tilt(\Rep_0(\textbf{G}))$, the tilting subcategory of the principal block of $\textbf{G}$. As explained in \cite[Rmk 5.2(1)]{rw}, we may now deduce the existence of a right $\mathscr{H}$-action on $\Rep_0(\textbf{G})$, and it remains only to verify that the push--pull action of $B_s$ maps across to the wall-crossing functor $\theta_s$.

We in fact show something more precise: functors of pushing and pulling in Smith theory correspond to translation functors onto and off walls in representation theory. In the case of pushforward, suppose $\gamma$ is the dominant coweight labelling an indecomposable tilting module $\mathsf{T}(\gamma)$ which affords the translation functor $T^s$ for $\textbf{G}$. On the tilting categories, the geometric Satake equivalence sends $T^s$ to the composite of convolution with a tilting sheaf $\mathscr{T}(\gamma)$ and a projection. Our first observation is that Smith localisation erases the difference between convolution with $\mathscr{T}(\gamma)$ and the functor $(\phi_2)_* (\phi_1)^*$, where the $\phi_i$ are the projections associated with a certain correspondence $\mathscr{Y}^\gamma \subseteq \Gr \times \Gr$. Our second observation is that since ``Smith localisation commutes with everything'', the following commutative diagram of geometric morphisms corresponds to a diagram of functors commuting up to natural isomorphism in Smith theory:
\begin{center}
\begin{tikzcd} 
 \Gr & \arrow[swap]{l}{\phi_1} \mathscr{Y}^\gamma \arrow{r}{\phi_2} & \Gr \\
 \Gr^\varpi \arrow[hookrightarrow]{u} & \arrow[swap]{l}{\phi_1^\varpi} (\mathscr{Y}^\gamma)^\varpi \arrow[hookrightarrow]{u} \arrow{r}{\phi_2^\varpi} & \Gr^\varpi \arrow[hookrightarrow]{u} \\
 \Fl_\ell \arrow[bend right=30,swap]{rr}{q^s} \arrow[hookrightarrow]{u} & \arrow[swap]{l}{z} \mathscr{Z} \arrow{u}{\theta} \arrow{r}{z^s} & \Fl_\ell^s \arrow[hookrightarrow]{u}.
 \end{tikzcd}
\end{center}
Here $\mathscr{Z}$ is merely the graph of $q^s$, so $(q^s)_* \cong z_*^s z^*$ and we are done; the case of pullback can be approached similarly or by citing properties of adjoint functors.

The structure of the paper is as follows. In Section \ref{sec:alg} we fix notation and cover necessary algebraic preliminaries. In Section \ref{sec:ing}, we provide context on the geometry of affine Grassmannians and flag varieties, (equivariant) derived categories, versions of the geometric Satake correspondence, and parity sheaves, among other topics. The purpose of Section \ref{sec:asm} is to recapitulate some of the main results of the third part of \cite{rw} with additional $\mathbb{G}_m$-equivariance. Section \ref{sec:asq} recalls the foundations of Smith--Treumann theory before constructing the aforementioned push--pull action of $\mathscr{H}$. Everything is tied together in Section \ref{sec:bri}, where the action is transported over to $\Rep_0(G)$ and, to conclude, the main results are stated (Theorems \ref{main} and \ref{final}).

\subsection*{Acknowledgements}
I am sincerely grateful to my supervisors, G. Williamson and O. Yacobi, for many helpful conversations, suggestions, and encouragements. This work was completed during PhD studies at the University of Sydney under the benefit of an RTP stipend. 

\section{Algebraic preliminaries} \label{sec:alg}
\subsection{Notation}
Let $\mathbb{F}$ and $\bk$ denote fields of unequal, positive characteristics $p$ and $\ell$, respectively, where $\mathbb{F}$ is algebraically closed and $\bk$ is finite. Write $\mathbb{O} = W(\bk)$ for the ring of Witt vectors over $\bk$, or in other words the unique unramified extension of $\mathbb{Z}_\ell$ of degree $[\bk:\mathbb{F}_\ell]$, and set $\mathbb{K}$ to be the quotient field of $\mathbb{O}$. We obtain an $\ell$-modular system:
\begin{equation} \label{lmod}
\bk \twoheadleftarrow \mathbb{O} \hookrightarrow \mathbb{K}.
\end{equation}
Throughout, all schemes will have ground field $\mathbb{F}$. We define a \textit{coefficient ring} to be one of those displayed in \eqref{lmod}, or a finite extension of $\bk$ or $\mathbb{K}$; these will provide the coefficients for {\'e}tale sheaves. All functors on derived categories will be assumed to be derived. For $n \ge 1$, let $\varpi_n \subseteq \mathbb{G}_m$ denote the subgroup of $n$-th roots of unity; put $\varpi = \varpi_\ell$. 

\subsection{Generalities on categories}
We collect some important notions from category theory, beginning with fundamentals on adjunctions. 

\begin{defn}
Let $\mathcal{C}, \mathcal{D}$ be categories and consider functors $F: \mathcal{D} \to \mathcal{C}$ and $G: \mathcal{C} \to \mathcal{D}$. 
\begin{enumerate}
    \item $F$ is \textit{left adjoint} to $G$, written $F \dashv G$, in case there is a natural isomorphism 
    $$\Phi: \text{Hom}_{\mathcal{C}}(F(-),-) \cong \text{Hom}_{\mathcal{C}}(-,G(-)): \mathcal{D} \times \mathcal{C}^\text{op} \to \text{Set}.$$ 
    Say $(F,G)$ is an \textit{adjoint pair} with \textit{adjunction} $\Phi$.
    \item Natural transformations $\varepsilon: FG \to 1_{\mathcal{C}}$, $\eta: 1_{\mathcal{D}} \to GF$ form a \textit{counit--unit pair} $(\varepsilon,\eta)$ for $(F,G)$ in case the \textit{counit} and \textit{unit} equations hold:
    $$\varepsilon F \circ F \eta = 1_F, \quad G \varepsilon \circ \eta G = 1_G.$$
\end{enumerate}
\end{defn}

\begin{lem} \label{adjeq}
Let $F: \mathcal{D} \to \mathcal{C}$ and $G: \mathcal{C} \to \mathcal{D}$. Then $(F,G)$ admits a counit--unit pair if and only if it is an adjoint pair.
\end{lem}

Indeed, if $\Phi$ is an adjunction for $(F,G)$, then we may define
$$\varepsilon_X = \Phi_{GX,X}^{-1}(1_{GX}), \quad \eta_Y = \Phi_{Y,FY}(1_{FY}).$$
The counit and unit equations now follow, using naturality of $\Phi$. Conversely, given a counit--unit pair $(\varepsilon,\eta)$, set $\Phi_{Y,X}(f) = G(f) \circ \eta_Y$, so that $\Phi_{Y,X}^{-1}(g) = \varepsilon_X \circ F(g).$ Notice that, for an adjoint pair $(F,G)$, any one of the three pieces of data $\Phi, \varepsilon, \eta$ is sufficient to determine the others.

Throughout the paper, certain equivariant derived categories will play a central role and we spend time constructing them below. For that discussion, we assume knowledge of triangulated categories, derived categories, and constructible and perverse sheaves on stratified topological spaces. That said, let us introduce two types of categorical quotient. Sources for the next definitions and results are \cite[\href{https://stacks.math.columbia.edu/tag/02MN}{Tag 02MN}, \href{https://stacks.math.columbia.edu/tag/05RA}{Tag 05RA}]{stacks-project}.

\begin{defn}
If $F: \mathcal{C} \to \mathcal{C}'$ is an additive functor between additive categories $\mathcal{C}$, $\mathcal{C}'$, then the \textit{kernel} of $F$ is the full subcategory
$$\ker F = \{ C \in \mathcal{C}': F(C) \cong 0 \} \subseteq \mathcal{C}.$$
\end{defn}

\begin{defn}
Let $\mathcal{A}$ be an abelian category. A non-empty full subcategory $\mathcal{C} \subseteq \mathcal{A}$ is a \textit{Serre subcategory} in case, for any exact sequence $A' \to A \to A''$ in $\mathcal{A}$, we have that $A \in \mathcal{C}$ if and only if $A', A'' \in \mathcal{C}$.
\end{defn}

\begin{prop}
Suppose $\mathcal{C}$ is a Serre subcategory of the abelian category $\mathcal{A}$. There exists an abelian category $\mathcal{A}/\mathcal{C}$ (the \textit{Serre quotient}) and an exact functor $$T: \mathcal{A} \to \mathcal{A}/\mathcal{C}$$ which are characterised by the following universal property: for any abelian category $\mathcal{B}$ and exact functor $S: \mathcal{A} \to \mathcal{B}$ with $\mathcal{C} \subseteq \ker S$, there is a unique exact functor $\overline{S}: \mathcal{A}/\mathcal{C} \to \mathcal{B}$ with $S = \overline{S} \circ T$.
\end{prop}

\begin{prop}
Suppose $\mathcal{C}$ is a full triangulated subcategory of a triangulated category $\mathcal{D}$. There exists a triangulated category $\mathcal{D}/\mathcal{C}$ (the \textit{Verdier quotient}) and a triangulated functor $$Q: \mathcal{D} \to \mathcal{D}/\mathcal{C}$$ which are characterised by the following universal property: for any triangulated category $\mathcal{E}$ and triangulated functor $P: \mathcal{D} \to \mathcal{E}$ with $\mathcal{C} \subseteq \ker P$, there is a unique triangulated functor $\overline{P}: \mathcal{D}/\mathcal{C} \to \mathcal{E}$ with $P = \overline{P} \circ Q$.
\end{prop}

We will make frequent and essential use of the notion of a direct limit of a system of categories, particularly in the case of derived categories over a directed family of varieties related by pushforward functors. We describe this in the framework of \cite[Appendix A]{waschkies}; see also \cite[\textsection 5.2]{brylinski}.


\begin{defn}
Let $I$ be a directed set and $\{ \mathcal{C}_i \}_{i \in I}$ a directed system of (essentially) small categories with transition functors $F_{ij}$ for $i \le j$ in $I$, satisfying $F_{ii} = \id$ and appropriate coherence conditions with respect to a family of natural isomorphisms,
$$\alpha_{ijk}: F_{jk} \circ F_{ij} \overset{\cong}{\to} F_{ik}.$$
The \textit{direct limit} of this system is the category $\mathcal{C} = \varinjlim_i \mathcal{C}_i$ with objects given by $$\text{Ob}(\mathcal{C}) = \bigsqcup_i \text{Ob}(\mathcal{C}_i),$$
and morphism spaces as follows: for $X \in \mathcal{C}_i$ and $Y \in \mathcal{C}_j$, let 
$$\Hom_{\mathcal{C}}(X,Y) = \varinjlim_{k \ge i,j} \Hom_{\mathcal{C}_k}(F_{ik}X,F_{jk}Y).$$
The transition maps for this colimit are 
$$\Hom_{\mathcal{C}_k}(F_{ik}X,F_{jk}Y) \overset{F_{kl}}{\to} \Hom_{\mathcal{C}_k}(F_{kl} F_{ik}X, F_{kl} F_{jk}Y) \cong \Hom_{\mathcal{C}_k}(F_{il}X,F_{jl}Y),$$
for $l \ge k \ge i,j$, where the latter isomorphism is afforded by $\alpha_{ikl}^{-1}$ and $\alpha_{jkl}$.
\end{defn}
Notice that we have natural \textit{injection functors}
$$F_i: \mathcal{C}_i \to \mathcal{C},$$
defined on objects by the inclusion of $\text{Ob}(\mathcal{C}_i)$ into the disjoint union of object collections, and on morphisms by the universal maps induced by colimits. Importantly, if the transition functors $F_{ij}$ are fully faithful for all $i \le j$, then so are the injection functors $F_i$. There are also natural isomorphisms 
$$\sigma_{ij}: F_i \overset{\cong}{\to} F_j \circ F_{ij},$$
given on each object $X$ by the class of the identity map $F_{ij}(X) \to F_{ij}(X)$ in $\Hom_\mathcal{C}(X,F_{ij}(X)) = \varinjlim_{k \ge j} \Hom(F_{ik} X, F_{jk} F_{ij} X)$. Equipped with the $F_i$ and the $\sigma_{ij}$, the category $\mathcal{C}$ satisfies a suitable universal property as a 2-colimit in the 2-category of small categories; see \cite[Def. A.2.1, Prop. A.3.6]{waschkies}.

Several of the main categories we will work with are of the following type. 

\begin{defn}
An additive category $\mathcal{A}$ is said to be a \textit{Krull--Schmdit category} in case every object decomposes into a finite direct sum of objects having local endomorphism rings. 
\end{defn}
If $\mathcal{A}$ is Krull--Schmidt, then the class of objects in $\mathcal{A}$ having local endomorphism rings coincides with the class of indecomposable objects. Moreover, every object splits into a direct sum of indecomposable objects in a unique fashion (up to isomorphism and reordering of summands). The following convenient criterion may be found in \cite[Corollary 3.3.3]{krause}.

\begin{prop}
If $\mathcal{A}$ is an abelian category all of whose objects are of finite length, then $\mathcal{A}$ is Krull--Schmidt.
\end{prop}

\subsection{Roots and Weyl groups}
Throughout, we assume $G$ is a semisimple algebraic group of adjoint type over $\mathbb{F}$. Fix $$T \subseteq B \subseteq G$$ a maximal torus and Borel subgroup of $G$, with $U$ the unipotent radical of $B$; let $B^+$ denote the opposite Borel subgroup to $B$, with unipotent radical $U^+$. Associated to this data is a root system
$$(\Phi \subseteq \textbf{X}, \Phi^\vee \subseteq \textbf{X}^\vee);$$
we write $\Phi_+ \subseteq \Phi$ for the positive roots opposite to $B$ and $\Sigma \subseteq \Phi_+$ for the simple roots. These give rise to a Coxeter group $(W_{\text{f}},S_{\text{f}})$, the finite Weyl group generated by finite simple reflections $s_\alpha$, $\alpha \in \Sigma$. Our assumptions on $G$ ensure the existence of an element $\rho^\vee \in \textbf{X}^\vee$ such that $\langle \alpha, \rho^\vee \rangle = 1$ for all $\alpha \in \Sigma$. We assume from now on that $\ell > h$, the Coxeter number of $\Phi$.

Considering the extended torus $\widetilde{T} = T \times \Gm$ gives us access to affine roots of the form $\alpha + m \delta \in X^*(\widetilde{T})$ for $m \in \mathbb{Z}$, where $\delta \in X^*(\widetilde{T})$ is the projection onto the factor $\Gm$. The affine Weyl group is
$$(W = W_{\text{f}} \ltimes \mathbb{Z} \Phi^\vee,S);$$
here the affine simple reflections $S$ comprise $S_{\text{f}}$ along with elements $t_{\beta^\vee} s_\beta$, where $\beta \in \Phi$ is maximal in the ordering determined by $\Phi_+$ and $t_{\beta^\vee}$ denotes the image of $\beta^\vee \in \mathbb{Z} \Phi^\vee$ in $W$. We consider the standard action of $W$ on $V = \textbf{X}^\vee \otimes_{\mathbb{Z}} \mathbb{R}$, along with the recentred and dilated actions
$$(wt_\nu) \bullet_n v = w(x + n\nu + \rho^\vee) - \rho^\vee, \quad (wt_\nu) \square_n v = w(x + n \nu).$$
\begin{rmk}
Our focus on $\textbf{X}^\vee$ arises from our intention to state representation-theoretic results for $\textbf{G}$, after making geometric arguments on the side of $G$.
\end{rmk}
\subsection{Blocks, translations, and tilting objects}
Write $\Rep(\textbf{G})$ for the abelian category of finite-dimensional algebraic representations of $\textbf{G}$. It is monoidal with the tensor product $\otimes$ over $\bk$, and its simple objects $\mathsf{L}(x)$ are parametrised by $x \in \textbf{X}_+^\vee$. 
The \textit{linkage principle} provides a decomposition of $\Rep(\textbf{G})$ into a direct sum of abelian subcategories,
$$\Rep(\textbf{G}) = \bigoplus_{c \in \textbf{X}^\vee/(W,\bullet_\ell)} \Rep_c(\textbf{G}),$$
where the \textit{block} $\Rep_c(\textbf{G})$ is the Serre subcategory generated by the $\mathsf{L}(x)$ with $x \in c \cap \textbf{X}_+^\vee$. (Note that while blocks of a category are often understood to be indecomposable, $\Rep_c(\textbf{G})$ need not be.) Now, let
$$C_{\mathbb{Z}} = \{ x \in \textbf{X}^\vee: \text{$0 < \langle \alpha, x + \rho^\vee \rangle < \ell$ for all $\alpha \in \Phi_+$} \}.$$
Since $\ell > h$, there is a \textit{regular weight} $\lambda_0 \in C_{\mathbb{Z}} \ne \varnothing$; we could take $\lambda_0 = 0$, but do not insist on it. For each $s \in S$, we let $\mu_s$ denote a \textit{subregular weight} lying on the reflection hyperplane of $s$ in $\overline{{C}_{\mathbb{Z}}}$, with respect to the $\bullet_\ell$-action of $W$, but no other such hyperplanes; these exist in our setting by \cite[\textsection II.6.3]{jan}. We then have the \textit{principal block} $\Rep_0(\textbf{G}) = \Rep_{[\lambda_0]}(\textbf{G})$ and the \textit{subregular blocks} $\Rep_s(\textbf{G}) = \Rep_{[\mu_s]}(\textbf{G}).$ 

An important theoretical role is played by \textit{translation functors} between the blocks. Specifically, we have translation functors onto and off the $s$-walls, $s \in S$:
$$T^s = T_{\lambda_0}^{\mu_s}, \quad T_s = T_{\mu_s}^{\lambda_0},$$
as defined in \cite[\textsection II.7]{jan}; up to natural isomorphism, these are independent from our choices of $\lambda_0$ and $\mu_s$, but depend in their definition on a non-canonical choice of module tensor factor. The functors $(T^s,T_s)$ are left and right adjoint to each other, and restrict to functors between the principal and subregular blocks:
$$T^s: \Rep_0(\textbf{G}) \to \Rep_s(\textbf{G}), \quad T_s: \Rep_s(\textbf{G}) \to \Rep_0(\textbf{G}).$$
Translation onto the $s$-wall then off it yields 
the \textit{wall-crossing functor} $\theta_s = T_s T^s$, a self-adjoint endofunctor of the principal block.

Our final recollection on $\Rep(\textbf{G})$ is that it has the structure of a \textit{highest weight category}, descending to all of its blocks, in the sense described in \cite{rw} and originally in \cite{cps}. In particular, there are \textit{standard} and \textit{costandard} objects $\Delta(x)$, $\nabla(x)$ in $\Rep(\textbf{G})$ (resp. $\Rep_0(\textbf{G})$, resp. $\Rep_s(\textbf{G})$) for $x \in \textbf{X}_+^\vee$ (resp. $x \in [\lambda_0] \cap \textbf{X}_+^\vee$, resp. $x \in [\mu_s] \cap \textbf{X}_+^\vee$), admitting morphisms
$$\Delta(x) \twoheadrightarrow \mathsf{L}(x) \hookrightarrow \nabla(x).$$
Objects which possess a filtration by standard objects and a filtration by costandard objects are said to be \textit{tilting}. These form additive (but not abelian) subcategories 
$$\Tilt \subseteq \Rep(\textbf{G}), \quad \Tilt_0 \subseteq \Rep_0(\textbf{G}), \quad \Tilt_s \subseteq \Rep_s(\textbf{G}),$$
all of which are closed under $\otimes$ and respected by the relevant translation functors. Furthermore, these categories are Krull--Schmidt, with the indecomposable tilting objects $\mathsf{T}(x)$ in $\Tilt$ (resp. $\Tilt_0$, resp. $\Tilt_s$) parametrised by $x \in \textbf{X}_+^\vee$ (resp. $x \in [\lambda_0] \cap \textbf{X}_+^\vee$, resp. $x \in [\mu_s] \cap \textbf{X}_+^\vee$).

\section{Geometric ingredients} \label{sec:ing}
\subsection{Loop groups and the affine Grassmannian}
Detailed treatments of the objects introduced in this subsection can be found in \cite{go}, \cite{kumar}, and \cite{zhu}; much of our notation follows \cite[\textsection 4]{st}. For $n \ge 1$, the \textit{$n$-th positive loop group} $L_n^+G$ of $G$ is the affine group scheme representing the functor from $\mathbb{F}$-algebras to sets given by
$$A \mapsto G(A [\![z^n]\!]).$$
It is a subfunctor of the \textit{loop group} $L_n G$, an ind-affine group ind-scheme representing the functor from $\mathbb{F}$-algebras to sets given by
$$A \mapsto G(A(\!(z^n)\!)).$$
We suppress $n$ from notation in case $n = 1$. Root subgroups $u_\alpha: \mathbb{G}_a \overset{\sim}{\to} U_\alpha \subseteq G$, for $\alpha \in \Phi$, give rise to \textit{affine root subgroups} $U_{\alpha + m \delta} \subseteq LG$, the images of morphisms described by the formula $x \mapsto u_\alpha(xz^m).$

Let $A$ be an $\mathbb{F}$-algebra. For any $a \in A^\times$, there is a map of $\mathbb{F}$-algebras
$$A(\!(z)\!) \to A(\!(z)\!), \quad z^m \mapsto a^m z^m.$$
This yields $G(A(\!(z)\!)) \to G(A(\!(z)\!))$, and hence a \textit{loop rotation} action of $\mathbb{G}_m$ on $LG$ stabilising $L^+G$. Since
$$(\Spec A(\!(z)\!))/{\varpi_n} \cong \Spec A(\!(z^n)\!),$$
we can identify $(LG)^{\varpi_n} \cong L_nG$ and $(L^+G)^{\varpi_n} \cong L_n^+G;$ see \cite[Lemma 4.2]{st}.

The \textit{affine Grassmanian} of $G$ is the ind-projective ind-scheme of ind-finite type $\Gr$ representing the fppf sheafification of the functor
$$A \mapsto (LG)(A)/(L^+G)(A).$$
For $\lambda \in \textbf{X}^\vee$, the image of $z$ under the mapping $\mathbb{F}(\!(z)\!)^\times \to G(\mathbb{F}(\!(z)\!))$ induced by $\lambda$ yields a point $z^\lambda$ in $LG$. If we denote by $L_\lambda$ the coset of $z^\lambda$ in $\Gr$, then every $L^+G$-orbit on $\Gr$ has the form $\Gr^\lambda = L^+G \cdot L_\lambda$ for $\lambda \in \textbf{X}_+^\vee$.

\subsection{Partial affine flag varieties} \label{s:flag}
Recall $V = \textbf{X}^\vee \otimes_{\mathbb{Z}} \mathbb{R}$ and fix $n \ge 1$ a positive integer. To each affine root $a = \alpha + m \delta$ we associate the affine function
$$f_a^n: V \to \mathbb{R}, \quad f_a^n(v) = \langle a, v \rangle + mn$$
and the family of affine reflections
$s_a = t_{m \alpha^\vee} s_\alpha,$ whose $\square_n$-action on $V$ is given by
$$s_a \square_n v = v - f_{\alpha-m \delta}^n(v) \alpha^\vee.$$
The zero sets of the affine functions, or equivalently the $\square_n$-fixed points of the affine reflections, form a hyperplane arrangement giving rise to a system of \textit{facets}. In particular we have
$$\textbf{a}_n = \{ \lambda \in V: \text{$-n < \langle \alpha, \lambda \rangle < 0$ for all $\alpha \in \Phi_+$} \},$$
an \textit{alcove} whose closure is a fundamental domain for the $\square_n$-action of $W$ on $V$. For every facet $\textbf{f} \subseteq \overline{\textbf{a}_n}$, Bruhat--Tits theory (as described in \cite[\textsection 4.2]{st}) provides a \textit{parahoric} group scheme $P^{\textbf{f}}$ over the ring $\mathcal{O}_n = \mathbb{F}[\![z^n]\!]$, such that
$P^{\textbf{a}_n} \subseteq P^{\textbf{f}}.$
Let $L_n^+ P^{\textbf{f}} \subseteq L_n^+ G$ be the affine subgroup scheme representing the functor
$$A \mapsto P^{\textbf{f}}(A[\![z]\!]).$$
We then have the \textit{partial affine flag variety} $\Fl_n^\textbf{f}$, the ind-projective ind-scheme of ind-finite type representing the fppf sheafification of the functor
$$A \mapsto (L_nG)(A)/(L_n^+P^{\textbf{f}})(A).$$
For $\lambda \in -\overline{\textbf{a}_n}$, denote by $\textbf{f}_\lambda \subseteq \overline{\textbf{a}_n}$ the facet containing $-\lambda$. Then the maps $$L_n G \to \Gr^{\varpi_n}, \quad g \mapsto g \cdot L_\lambda$$
factor through embeddings $\Fl_n^{\textbf{f}_\lambda} \hookrightarrow \Gr^{\varpi_n}$. Their images $\Gr_{(\nu)}^n$ feature in the beautiful and crucially important decomposition of fixed points \cite[Proposition 4.6]{st}:
$$\Gr^{\varpi_n} \cong \bigsqcup_{\nu \in -\overline{a_n} \cap \textbf{X}^\vee} \Gr_{(\nu)}^n, \quad \text{where} \quad \Gr_{(\nu)}^n \cong \Fl_n^{\textbf{f}_\nu} \overset{j_{(\nu)}^n}{\longhookrightarrow} \Gr^{\varpi_n}.$$
In the context of $\varpi_n$-fixed points, omission of $n$ from notation will correspond to the case $n = \ell$, so we will write e.g. $\Gr_{(\nu)} = \Gr_{(\nu)}^\ell$ and $j_{(\nu)} = j_{(\nu)}^\ell$.

Our main interest will be in the following special affine flag varieties. Firstly, note that evaluation at zero $A [\![z^n]\!] \to A$ yields a morphism $\text{ev}_n: L_n^+G \to G$. The \textit{$n$-th Iwahori group} is $I_n = \text{ev}_n^{-1}(B);$ it coincides with the positive loop group $L_n^+ P^{\textbf{a}_n}$. The corresponding partial affine flag variety is written $\Fl_n$ and known simply as the \textit{$n$-th affine flag variety} of $G$; it admits a decomposition into $I_n$-orbits,
$$\Fl_n = \bigsqcup_{x \in W} \Fl_{n,x}, \quad \text{where $\Fl_{n,x} = I_n \cdot x I_n/I_n.$}$$
We let $I_{n,u} = \text{ev}_n^{-1}(U)$ denote the pro-unipotent radical of $I_n$. Replacing $B$ (resp. $U$) with $B^+$ (resp. $U^+$) yields opposite groups
$$I_n^+ = w_0 I_n w_0, \quad I_{n,u}^+ = w_0 I_{n,u} w_0.$$
Note that the $L^+G$-orbits on $\Gr$ decompose into finitely many Iwahori orbits,
$$\Gr^\lambda = \bigsqcup_{\mu \in W_{\text{f}}(\lambda)} \Gr_\mu, \quad \text{where $\Gr_\mu = I \cdot L_\mu = I_u \cdot L_\mu$},$$
and an analogous statement hold for $I^+$ and $I_u^+$.

Secondly, observe that $\mu = \mu_s + \rho^\vee \in -\overline{a}_\ell \cap \textbf{X}^\vee$, so we have a parahoric group scheme $P^{\textbf{f}_s}$ and positive loop group $\mathcal{P}_n^s = L_n^+ P^{\textbf{f}_s}$ for the facet $\textbf{f}_s = \textbf{f}_{\mu_s+\rho^\vee}$; we write $\Fl_n^s = \Fl_n^{\textbf{f}_s}$. If $s \in W_{\text{f}}$ is a finite simple reflection, then $\mathcal{P}_n^s$ is the inverse image $\text{ev}_n^{-1}(P_s)$ of the standard parabolic subgroup $P_s \subseteq G$ containing $B$. In any case,
$I_n \subseteq \mathcal{P}_n^s$ and there is a natural proper morphism $q_n^s: \Fl_n \to \Fl_n^s.$
This morphism is a $\mathbb{P}_{\mathbb{F}}^1$-bundle and hence such that $(q_n^s)^! \cong (q_n^s)^*[2]$; see \cite[8.e.1, Prop. 8.7]{pr}.  

Until further notice, we assume the loop index $n = 1$ to simplify notation.
\begin{prop}
$L^+G$ and $I$ are stable under the action of $\mathbb{G}_m$ on $LG$. 
\end{prop}
\begin{proof}
This follows directly from the observation that $\text{ev}: L^+G \to G$ intertwines the action of $\mathbb{G}_m$ (letting $\mathbb{G}_m$ act trivially on $G$).
\end{proof}

Accordingly, we obtain a $\mathbb{G}_m$-action on $\Fl$. Since it stabilises $I$, the action also preserves $I$-orbits. In particular, if $X \subseteq \Fl$ is a locally closed finite union of strata $\Fl_x$, then $X$ admits an action of $I \rtimes \mathbb{G}_m$.

\subsection{Equivariant derived categories on partial affine varieties}
At all times in the sequel, our schemes will be defined over $\mathbb{F}$ and we will work in the ``{\'e}tale context'' described in \cite[\textsection 9.3(2)]{rw}, referring to (possibly equivariant) derived categories of {\'e}tale sheaves over a coefficient ring $\mathbb{L}$. We discuss the equivariant derived category as introduced in \cite{bl} for the topological setting, but note that the necessary adjustments for {\'e}tale sheaves are provided in \cite{weidner}. 
 
\begin{lem} \label{finquot}
Suppose $L$ is a connected ind-affine group ind-scheme acting on a projective $\mathbb{F}$-variety $P$ of finite type. Then the action of $L$ on any locally closed $L$-stable subvariety $X \subseteq P$ factors through a quotient group scheme of finite type. 
\end{lem}

\begin{proof}
Since $P$ is a projective variety over a field, $\text{Aut}(P)$ is naturally a scheme with neutral component $\text{Aut}^\circ(P)$ of finite type \cite{brion}. Recalling that $L$ is assumed to be connected, we have a group homomorphism $\varphi: L \to \text{Aut}^\circ(P)$. Now the action of $L$ on $P$ factors through the image of $\varphi$, and the action on $X$ is then obtained by restriction.
\end{proof}

Suppose $X$ is a locally closed finite union of $I$-orbits in $\Fl$ (resp. $\Fl^s$). By Lemma \ref{finquot}, $I \subseteq L^+ G$ acts on $X$ through a quotient group scheme $\Gamma$ of finite type. Since $I \cong I_{0} \rtimes G$, where $$I_{0} = \ker(\text{ev}: L^+G \to G)$$ is the \textit{pro-unipotent radical} of $I$, we are free to insist that the kernel of $I \to \Gamma$ is contained in $I_{0}$. Such a finite quotient $\Gamma$ will be called \textit{suitable}. Through the construction in \cite{bl}, we may consider the $\Gamma$-\textit{equivariant derived category} $D_\Gamma^b(X,\L)$. As mentioned in \cite[A.4]{br} and the appendix to \cite{gait}, $D_\Gamma^b(X,\L)$ is independent of the choice of suitable $\Gamma$: two suitable groups $\Gamma$ and $\Gamma'$ admit a common suitable refinement $\Gamma''$, and group surjections with unipotent kernels induce equivalences of equivariant derived categories:
$$D_\Gamma^b(X,\L) \cong D_{\Gamma''}^b(X,\L) \cong D_{\Gamma'}^b(X,\L).$$
Thus we can safely define 
$$D_I^b(X,\L) = D_\Gamma^b(X,\L)$$
for any locally closed finite union of strata $X$ and any suitable $\Gamma$ acting on $X$. Notice that every inclusion $X \hookrightarrow Y$ of locally closed subsets induces a fully faithful pushforward functor
$$D_I^b(X,\L) \to D_I^b(Y,\L).$$
Thus we can take a direct limits over those $X$ which are closed to obtain a triangulated equivariant derived category $D_I^b(\Fl,\L)$ (resp. $D_I^b(\Fl^s,\L)$). By a similar process we may construct $D_{I \rtimes \mathbb{G}_m}^b(\Fl,\L)$ (resp. $D_{I \rtimes \mathbb{G}_m}^b(\Fl^s,\L)$). 

\subsection{(Loop rotation equivariant) Iwahori--Whittaker derived categories} \label{s:iw}
The material in this section draws from \cite[Appendix A]{ar} and \cite[\textsection 5.1--5.2]{st}. Let $\mathscr{X}$ denote either the affine Grassmannian $\Gr$ or a partial affine flag variety such as $\Fl$ and $\Fl^s$. Assume there exists a non-trivial $p$-th root of unity in $\zeta \in \bk$, then let
$$\tau: \mathbb{G}_a \to \mathbb{G}_a, \quad x \mapsto x^p - x,$$
be the \textit{Artin--Schreier map}; this is a Galois covering with Galois group $\mathbb{F}_p$. We define the associated \textit{Artin--Schreier local system} $\mathscr{L}_{\text{AS}}$ to be the summand of $\tau_* \underline{\bk}_{\mathbb{G}_a}$ on which $\mathbb{F}_p$ acts by powers of $\zeta$. Finally, let $$\chi = \chi_0 \circ \text{ev}: I_u^+ \to U^+ \to \mathbb{G}_a,$$ where $\chi_0: U^+ \to \mathbb{G}_a$ is a fixed morphism of algebraic groups which is non-trivial on any simple root subgroup of $U^+$. 

If $X \subseteq \mathscr{X}$ is a locally closed finite union of $I^+$-orbits, then $I_u^+$ acts on $X$ through some finite quotient $J$ of $I_u^+$, which can be chosen in such a way that $\chi$ factors through a morphism $\chi_J: J \to \mathbb{G}_a$. We can then consider the $(J,\chi_J^* \mathscr{L}_{\text{AS}})$-equivariant derived category $D_{\IW}^b(X,\bk) = D_{J,\chi}^b(X,\bk)$: this is the full subcategory of $D_c^b(X,\bk)$ whose objects $F$ are such that 
\begin{equation} \label{iwcond}
    a_J^* F \cong \chi_J^* \mathscr{L}_{\text{AS}} \boxtimes F,
\end{equation}
where $a_J: J \times X \to X$ is the action map. As with previous constructions, this category is independent of our choices (up to equivalence). Importantly, we have an essentially surjective \textit{averaging functor}
\begin{equation} \label{avg}
D_c^b(X,\bk) \to D_{\IW}^b(X,\bk), \quad F \mapsto (a_J)_!(\chi_J^* \mathscr{L}_{\text{AS}} \boxtimes F).
\end{equation} 
Taking a direct limit over closed $X$, we obtain a triangulated
\textit{Iwahori--Whittaker category} $D_{\IW}^b(\mathscr{X},\bk)$. This category has a natural perverse t-structure with heart $\Perv_{\IW}(\mathscr{X},\bk)$, and admits a fully faithful forgetful functor
$$D_{\IW}^b(\mathscr{X},\bk) \to D_c^b(\mathscr{X},\bk).$$

Assume now that the action of $I^+$ on $\mathscr{X}$ extends to an action of $I^+ \rtimes \mathbb{G}_m$. Then, for each $X$ as above, the finite-type quotient $J$ of $I_u^+$ can be chosen in such a way that the action of $\mathbb{G}_m$ on $I^+$ descends to an action on $J$. We can hence consider $D_{\IW, \mathbb{G}_m}^b(X,\bk)$, the full subcategory of $D_{\mathbb{G}_m,c}^b(X,\bk)$ whose objects $F$ satisfy the analogue of the condition \eqref{iwcond} in $D_{\mathbb{G}_m}^b(J \times X,\bk)$; here $\mathbb{G}_m$ is assumed to act on $J \times X$ diagonally. 

As with the previous construction, $D_{\IW, \mathbb{G}_m}^b(X,\bk)$ is triangulated with a natural t-structure, and is independent of our choices up to equivalence. Its heart will be written $\Perv_{\IW, \mathbb{G}_m}(X,\bk)$.
Taking a direct limit over those $X$ which are closed, we obtain a triangulated \textit{loop rotation equivariant Iwahori--Whittaker category} $D_{\IW, \mathbb{G}_m}(\mathscr{X},\bk)$, with heart $\Perv_{\IW, \mathbb{G}_m}(\mathscr{X},\bk)$ with respect to the inherited t-structure. There is a functor of forgetting $\mathbb{G}_m$-equivariance,
$$D_{\IW,\mathbb{G}_m}^b(\mathscr{X},\bk) \to D_{\IW}^b(\mathscr{X},\bk).$$

\subsection{Convolution products} \label{s:conv}
We will need a \textit{convolution product} on $D_{I \rtimes \mathbb{G}_m}^b(\mathscr{X},\L)$ in cases where $\mathscr{X}$ is the affine Grassmannian or a partial affine flag variety. Given $E, F \in D_{I \rtimes \mathbb{G}_m}^b(\mathscr{X},\L)$, choose a locally closed finite union of $I$-orbits $X \subseteq \mathscr{X}$ from which $F$ is pushed forward, and suppose $J$ is a suitable quotient of $I \rtimes \mathbb{G}_m$ expressing its action on $X$: 
$$1 \to H \to I \rtimes \mathbb{G}_m \to J \to 1$$
We can then consider the diagram
$$\mathscr{X} \times X \leftarrow (LG \rtimes \mathbb{G}_m)/H \times X \overset{q}{\to} (LG \rtimes \mathbb{G}_m)/H \times^{J} X \to \mathscr{X}.$$
The morphism $q$ induces an equivalence
$$q^*: D_J^b((LG \rtimes \mathbb{G}_m)/H \times^J X,\L) \to D_{J \times J}^b((LG \rtimes \mathbb{G}_m)/H \times X,\L),$$
by \cite[Theorem 2.6.2]{bl}, since the left action of $J$ on $(LG \rtimes \mathbb{G}_m)/H$ is free. Thus, up to isomorphism, there is a unique object $E \widetilde{\boxtimes} F$ such that
$$q^*(E \widetilde{\boxtimes} F) \cong p^*(E) \boxtimes F,$$
where $p: (LG \rtimes \mathbb{G}_m)/H \to \mathscr{X}$ is the projection. We then define 
$$E \star F = m_*(E \widetilde{\boxtimes} F) \in D_{I \rtimes \mathbb{G}_m}(\mathscr{X},\L).$$
Similarly, we have
$$
D_{L^+G}^b(\Gr,\L) \curvearrowleft D_{L^+G}^b(\Gr,\L), \quad D_{\IW}^b(\Gr,\bk) \curvearrowleft D_{L^+G}^b(\Gr,\bk),$$
$$D_{\IW,\Gm}^b(\Gr,\bk) \curvearrowleft D_{L^+G}^b(\Gr,\bk), \quad D_{\IW,\Gm}^b(\mathscr{X},\bk) \curvearrowleft D_{I \rtimes \Gm}^b(\mathscr{X},\bk),$$
where the symbol $\curvearrowleft$ indicates the second category is right-acting on the first category via convolution.



\subsection{Parity sheaves} \label{s:par}
Again let $X \subseteq \mathscr{X}$ be a locally closed finite union of $I^+$-orbits. Using forgetful functors to $D_c^b(X,\L)$, we have a notion of \textit{parity complexes} in $D_{I \rtimes \mathbb{G}_m}^b(X,\L)$ and $D^b_{\IW, \mathbb{G}_m}(X,\bk)$; the main reference for these is \cite{jmw}. In the direct limit, we obtain full subcategories
$$\Parity_{I \rtimes \mathbb{G}_m}(\mathscr{X},\L) \subseteq D^b_{I \rtimes \mathbb{G}_m}(\mathscr{X},\L),$$
$$\Parity_{\IW}(\mathscr{X},\bk) \subseteq D^b_{\IW}(\mathscr{X},\bk), \quad \Parity_{\IW, \mathbb{G}_m}(\mathscr{X},\bk) \subseteq D^b_{\IW, \mathbb{G}_m}(\mathscr{X},\bk).$$
Suppose now that $\mathscr{X}$ has a Bruhat-like decomposition
\begin{equation} \label{bruhat}
\mathscr{X} = \bigsqcup_{\alpha \in A} \mathscr{X}_\alpha,
\end{equation}
where the $\mathscr{X}_\alpha$ are $I$-orbits of dimensions $d_\alpha$. For reasons articulated in \cite[\textsection 5.3]{st} and building on the general theory of \cite{jmw}, there is (up to isomorphism and shift) at most one indecomposable parity sheaf
$$\mathscr{E}^\mathscr{X}(\alpha) \in \Parity_{I \rtimes \Gm}(\mathscr{X},\L)$$
which is supported on $\overline{\mathscr{X}_\alpha}$ and whose restriction to $\mathscr{X}_\alpha$ is $\underline{\L}_{\mathscr{X}_\alpha}[d_\alpha]$. In the cases where $\mathscr{X}$ is $\Fl$, resp. $\Fl^s$, we have $A = W$, resp. $A = W^s$, and we write $\mathscr{E}^\Fl(w) = \mathscr{E}(w)$, resp. $\mathscr{E}^{\Fl^s}(w) = \mathscr{E}^s(w)$; sometimes a subscript $\L$ will be included to emphasise the ground ring.

On the other hand, suppose $\mathscr{X}$ admits a similar decomposition
\begin{equation} \label{bruhat2}
\mathscr{X} = \bigsqcup_{\alpha \in A^+} \mathscr{X}_\alpha^+,
\end{equation}
where the $\mathscr{X}_\alpha^+$ are $I^+$-orbits of dimensions $d_\alpha^+$. Write $A_{+} \subseteq A^+$ for the subset parametrising the orbits $\mathscr{X}_\alpha$ that support a non-zero Iwahori--Whittaker local system $\mathscr{L}_{\text{AS}}^{\mathscr{X}}(\alpha) \in D_{\IW,\mathbb{G}_m}^b(\mathscr{X}_\alpha^+,\bk)$, which is necessarily unique up to isomorphism. In a similar fashion to the previous case, there is (up to isomorphism and shift) at most one parity sheaf $$\mathscr{E}_{\IW,\mathbb{G}_m}^{\mathscr{X}}(\alpha) \in \Parity_{\IW, \mathbb{G}_m}(\mathscr{X},\bk), \quad \text{resp.} \quad \mathscr{E}_{\IW}^{\mathscr{X}}(\alpha) \in \Parity_{\IW}(\mathscr{X},\bk),$$ supported on $\overline{\mathscr{X}_\alpha^+}$ and extending $\mathscr{L}_{\text{AS}}[d_\alpha^+]$, for each $\alpha \in A_+$. Moreover, the former of these is sent to the latter under the forgetful functor
$$\Parity_{\IW,\mathbb{G}_m}(\mathscr{X},\bk) \to \Parity_{\IW}(\mathscr{X},\bk).$$
The following will be our main examples:
\begin{enumerate}
    \item $\mathscr{X} = \Gr$, where $(A_+,A^+) = (\textbf{X}_{+\!+}^\vee, \textbf{X}^\vee)$.
    \item $\mathscr{X} = \Fl$, where $(A_+,A^+) = (^{\text{f}}W, W)$ and we write 
    $$\mathscr{E}_{\IW, \mathbb{G}_m}(w) := \mathscr{E}_{\IW, \mathbb{G}_m}^\Fl(w), \quad \mathscr{E}_{\IW}(w) := \mathscr{E}_{\IW}^\Fl(w).$$
    \item $\mathscr{X} = \Fl^s$, where $(A_+,A^+) = (^{\text{f}}W^s, W^s)$ for
    $$W^s = \{ w \in W: w < ws \}, \quad {^{\text{f}}}W^s = \{ w \in W^s \cap {^{\text{f}}}W: ws \in {^{\text{f}}}W \},$$
    and we write 
    $$\mathscr{E}_{\IW, \mathbb{G}_m}^s(w) := \mathscr{E}_{\IW, \mathbb{G}_m}^{\Fl^s}(w), \quad \mathscr{E}_{\IW}^s(w) := \mathscr{E}_{\IW}^{\Fl^s}(w).$$
\end{enumerate}

\subsection{Highest weight objects and averaging}
Assume again that $\mathscr{X}$ has a decomposition $\eqref{bruhat2}$, with affine embeddings
$$j_\alpha: \mathscr{X}_\alpha \hookrightarrow \mathscr{X}.$$
Then for $\alpha \in A^+$ we have \textit{standard} and \textit{costandard} objects
$$\Delta_{\IW,\mathbb{G}_m}^\mathscr{X}(\alpha) = (j_\alpha)_! \mathscr{L}_{\text{AS}}^\mathscr{X}(\alpha)[d_\alpha], \quad \nabla_{\IW,\mathbb{G}_m}^\mathscr{X}(\alpha) = (j_\alpha)_* \mathscr{L}_{\text{AS}}^\mathscr{X}(\alpha)[d_\alpha]$$
in $D_{\IW,\mathbb{G}_m}^b(\mathscr{X},\bk)$, whose images under the forgetful functor to $D_{\IW}^b(\mathscr{X},\bk)$ are written $\Delta_{\IW}^\mathscr{X}(\alpha)$, $\nabla_{\IW}^\mathscr{X}(\alpha)$, respectively. We adopt notational abbreviations as in Section \ref{s:par} for the cases where $\mathscr{X}$ is a partial affine flag variety, writing e.g. 
$$\Delta_{\IW}(\alpha) = \Delta_{\IW}^\Fl(\alpha) \quad \text{and} \quad \Delta_{\IW}^s(\alpha) = \Delta_{\IW}^{\Fl_s}(\alpha).$$
For $\mathscr{X} = \Fl$, we have an \textit{averaging functor}
$$\mathsf{Av}: D_{I \rtimes \mathbb{G}_m}^b(\Fl,\bk) \to D_{\IW,\mathbb{G}_m}^b(\Fl,\bk), \quad F \mapsto \Delta_{\IW,\Gm}(1) \star F;$$
this definition agrees with \eqref{avg} on locally closed finite unions of orbits.

\subsection{Extension of scalars} \label{s:ext}
Recall from \cite[\textsection 10.2]{rw} that any extension $\mathbb{L} \to \mathbb{L}'$ of coefficient rings induces a monoidal \textit{extension of scalars} functor
$$\mathbb{L}'(-) = \mathbb{L}' \otimes_\mathbb{L} (-): D_{I \rtimes \Gm}^b(\mathscr{X},\mathbb{L}) \to D_{I \rtimes \Gm}^b(\mathscr{X},\mathbb{L}')$$
which preserves parity subcategories; see \cite[Lemma 2.36]{jmw}. Extension of scalars also affords isomorphisms
\begin{equation} \label{extsc}
\mathbb{L}' \otimes_{\mathbb{L}} \Hom_{\Parity_{I \rtimes \Gm}(\mathscr{X},\mathbb{L})}(\mathscr{E},\mathscr{F}) \overset{\cong}{\to} \Hom_{\Parity_{I \rtimes \Gm}(\mathscr{X},\mathbb{L}')}(\mathbb{L}'(\mathscr{E}),\mathbb{L}'(\mathscr{F})),
\end{equation}
and $\mathbb{L}'(-)$ is compatible with pushforward and pullback along morphisms such as $q^s$. In the case of $\O \to \K$, we deduce an $\O$-module injection
$$\Hom_{\Parity_{I \rtimes \Gm}(\mathscr{X},\mathbb{O})}(\mathscr{E},\mathscr{F}) \hookrightarrow \Hom_{\Parity_{I \rtimes \Gm}(\mathscr{X},\mathbb{K})}(\K(\mathscr{E}),\K(\mathscr{F})),$$
since the source is torsion free over $\O$; see \cite[Lemma 2.2(2)]{mr}. In the case of $\mathbb{O} \to \bk$, it holds that $\bk(\mathscr{E})$ is a parity sheaf if and only if $\mathscr{E}$ is; moreover, by \cite[Prop. 2.39]{jmw}, extension to $\bk$ respects indecomposable parity objects: 
\begin{equation} \label{bc}
\bk(\mathscr{E}_{\mathbb{O}}^{\mathscr{X}}(\alpha)) \cong \mathscr{E}_{\mathbb{O}}^{\mathscr{X}}(\alpha).
\end{equation}

\subsection{Geometric Satake equivalence}
We briefly summarise the contents of \cite[8.1]{st}, for the sake of fixing notation and recalling key results. Denote by $$\Perv(\Gr,\bk) \subseteq D_{L^+G}^b(\Gr,\bk)$$ the \textit{Satake category} of $L^+ G$-equivariant perverse sheaves on $\Gr$; this category inherits an exact monoidal convolution product $\star$ from $D_{L^+G}^b(\Gr,\bk)$, while its simple objects $\IC(\lambda)$ are parametrised by $\lambda \in \textbf{X}_+^\vee$. Our main application for perverse sheaves derives from the following theorem, due to Mirkovi\'c--Vilonen \cite{mv07}.
\begin{thm}
There is an equivalence of monoidal categories,
$$\Satake: (\Rep(\textbf{G}),\otimes) \cong (\Perv(\Gr,\bk),\star).$$
\end{thm}
In fact, even more important for us will be a formulation of this theorem featuring the Iwahori--Whittaker derived category; this is made possible by the following result of \cite{bgmrr}.
\begin{thm}
There is an equivalence of abelian categories,
$$\Perv(\Gr,\bk) \cong \Perv_{\IW}(\Gr,\bk), \quad \mathscr{F} \mapsto \Delta_\IW^\Gr(\rho^\vee) \star \mathscr{F}.$$
\end{thm}
It will also be important that the forgetful functor 
$$\For_{\Gm}: \Perv_{\IW,\Gm}(\Gr,\bk) \to \Perv_\IW(\Gr,\bk)$$ is an equivalence of categories; this is shown in \cite[Lemma 5.2]{st}.

Since $\Perv(\Gr,\bk)$ has the structure of a highest weight category, we may refer to its tilting subcategory $\Tilt(\Gr,\bk) \subseteq \Perv(\Gr,\bk)$, whose indecomposable tilting objects $\mathscr{T}(x)$ are indexed by $x \in \textbf{X}_+^\vee$. We likewise have $$\Tilt_{\IW}(\Gr,\bk) \subseteq \Perv_{\IW}(\Gr,\bk), \quad \Tilt_{\IW,\Gm}(\Gr,\bk) \subseteq \Perv_{\IW}(\Gr,\bk)$$ 
whose indecomposable tilting objects are indexed by $x \in \textbf{X}_{+\!+}^\vee$ and written $\mathscr{T}_{\IW}(x)$ and $\mathscr{T}_{\IW,\Gm}(x)$, respectively.

\section{The loop antispherical module} \label{sec:asm}
\subsection{Bott--Samelson varieties}
Let $s \in S$ be a simple affine reflection and recall from \textsection \ref{s:flag} the positive loop group of the parahoric group scheme associated to $s$: 
$$\mathcal{P}^s = L^+ P_{\textbf{f}_s}.$$ 
For any expression $\underline{w} = (s_1, s_2, \dots, s_m)$ in $S$, there is a \textit{Bott--Samelson variety} 
$$\nu_{\underline{w}}: \BSvar(\underline{w}) = \mathcal{P}^{s_1} \times^{I} \mathcal{P}^{s_2} \times^{I} \cdots \times^{I} \mathcal{P}^{s_k}/I \to \Fl.$$
As in \cite[\textsection 9.1]{rw}, the morphism $\nu_{\underline{w}}$ is equivariant for the natural $I$-action on its source and target; since it arises from the multiplication
$$\mathcal{P}^{s_1} \times \cdots \times \mathcal{P}^{s_k} \to LG,$$
$\nu_{\underline{w}}$ is also $\mathbb{G}_m$-equivariant, and thus equivariant for the action of $I \rtimes \mathbb{G}_m$. Moreover, $\BSvar(\underline{w})$ is a smooth projective variety, so $\nu_{\underline{w}}$ is quasi-separated; hence, $(\nu_{\underline{w}})_*$ preserves $I \rtimes \mathbb{G}_m$-equivariant sheaves under pushforward. Consider then the object
$$\mathscr{E}(\underline{w}) = (\nu_{\underline{w}})_* \underline{\L}_{\BSvar(\underline{w})}[\ell(\underline{w})] \in D_{I \rtimes \mathbb{G}_m}^b(\Fl,\L).$$
An easy adaption of the proofs of \cite[Lemma 10.2]{rw} shows there are canonical isomorphisms
\begin{equation} \label{concat}
\mathscr{E}(\underline{w}) \star^{I \rtimes \mathbb{G}_m} \mathscr{E}(\underline{v}) \cong \mathscr{E}(\underline{wv}).
\end{equation}
The smoothness of $\mathcal{P}^s/I$ implies $\mathscr{E}(s)$ is a parity object, and convolution by $\mathscr{E}(s)$ preserves parity objects for the reasons given in \cite[\textsection 9.4]{rw}, so due to \eqref{concat} all the $\mathscr{E}(\underline{w})$ are parity objects.
\begin{defn}
We consider a category $\Parity_{I \rtimes \mathbb{G}_m}^{\BSvar}(\Fl,\L)$ whose objects are pairs $(\underline{w},m)$ for $\underline{w}$ an expression and $m$ an integer, and whose morphism spaces
are 
$$\text{Hom}((\underline{w},m),(\underline{v},n)) = \text{Hom}(\mathscr{E}(\underline{w})[m],\mathscr{E}(\underline{v})[n]).$$
It is monoidal with $(\underline{w},m) \star (\underline{v},n) = (\underline{wv}, m + n)$ on objects, and the product of morphisms defined through (\ref{concat}).
\end{defn}

As in \cite[Lemma 10.4]{rw}, the natural functor
$$\Parity_{I \rtimes \mathbb{G}_m}^\BSvar(\Fl,\L) \to \Parity_{I \rtimes \mathbb{G}_m}(\Fl,\L)$$
realises $\Parity_{I \rtimes \mathbb{G}_m}(\Fl,\L)$ as the Karoubi envelope of the additive hull of $\Parity_{I \rtimes \mathbb{G}_m}^\BSvar(\Fl,\L)$. Note also that $\Parity_{I \rtimes \mathbb{G}_m}(\Fl,\L)$ is Krull--Schmidt, as a full subcategory of the Krull--Schmidt category $D_{I \rtimes \mathbb{G}_m}^b(\Fl,\L)$; see \cite[Remark 2.1]{jmw} and \cite{lc}. 

\subsection{Realisations and Bott--Samelson Hecke categories}
We begin by describing two realisations of the Coxeter system $(W,S)$, in the sense of \cite[\textsection 3.1]{ew}. For simplicity, we assume $\Phi$ is irreducible; the discussion can be made fully general by considering components in turn.

Consider first the $\L$-module
$$\mathfrak{h}_\L = \L \otimes_{\mathbb{Z}} \mathfrak{h}_{\mathbb{Z}}, \quad \text{where} \quad \mathfrak{h}_{\mathbb{Z}} = \mathbb{Z} \Phi^\vee \oplus \mathbb{Z} d.$$
In $\mathfrak{h}_\L$, we have the simple coroots
$\alpha_1^\vee, \dots, \alpha_l^\vee \in \Phi^\vee,$
along with $\alpha_0^\vee = -\theta^\vee$, where $\theta$ is the highest root in $\Phi_+$; in $\mathfrak{h}_\L^*$, there are simple roots $\alpha_1, \dots, \alpha_l$,
along with $\alpha_0 = \delta - \theta$, where $\delta$ is the projection onto $\mathbb{Z}d$. We define reflections 
$$\sigma_{i}: \mathfrak{h}_\L \to \mathfrak{h}_\L, \quad \sigma_{i}(v) = v - \langle v, \alpha_i \rangle \alpha_i^\vee,$$
for $0 \le i \le \ell$; note here that $\langle d, \alpha_i \rangle = \delta_{0i}$. Now, the data $(\mathfrak{h}_\L, \{ \alpha_i^\vee \}_{0 \le i \le l}, \{ \alpha_i \}_{0 \le i \le l})$, together with the assignment of the simple reflections in $W$ to the $\sigma_i$, define the \textit{loop realisation} of $(W,S)$. In fact, this realisation is a quotient of the ``traditional'' realisation associated to the affinisation of the Lie algebra of $G$, discussed in \cite[Remark 10.17(2)]{rw}.

Secondly, consider the realisation of $(W,S)$ described in \cite[\textsection 4.2]{rw} (for the dual root system); there it is denoted $\mathfrak{h}$, but we will denote it $\mathfrak{h}_\L'$ and call it the \textit{standard realisation}. We have
$$\mathfrak{h}_\L' = \L \otimes_{\mathbb{Z}} \mathbb{Z} \Phi^\vee \hookrightarrow \mathfrak{h}_\L,$$
with the same simple coroots and simple roots as $\mathfrak{h}_\L$, except with $\alpha_0' = -\theta$ in place of $\alpha_0$. Clearly the standard and loop realisations have the same Cartan matrix; hence by \cite[\textsection 4.2]{rw}, both realisations satisfy the technical conditions in \cite[Definition 3.6]{ew} and \cite[Assumption 3.7]{ew} if $\ell > h$.

Now let $$R_\L' = \text{Sym}((\mathfrak{h}_\L')^*), \quad R_\L = \text{Sym}(\mathfrak{h}_\L^*) = R_\L'[d],$$ 
which are generated as $\L$-algebras by the simple roots in $\mathfrak{h}_\L$ and $\mathfrak{h}_\L'$, respectively. The algebra $R_\L'$, respectively $R_\L$, is essential to define the graded $\bk$-linear monoidal \textit{Bott--Samelson Hecke category} $\mathscr{H}_{\BSvar}'$, respectively $\mathscr{H}_{\BSvar}$, associated to the standard realisation, respectively the loop realisation. For these definitions we refer the reader to \cite[Definition 5.2]{ew}, which features the notation $\mathcal{D}$ (suppressing mention of the realisation). The standard generating objects, shift of grading, and monoidal product in Hecke categories will be denoted by the symbols $B_s$, $\langle 1 \rangle$, and $\star$, respectively.

\subsection{First equivalences} \label{s:first}
It is proven in \cite[\textsection 10]{rw} that there is an equivalence of $\L$-linear graded monoidal categories,
$$\mathscr{H}_{\BSvar}' \cong \Parity_{I}^{\BSvar}(\Fl,\L),$$
lifting functorially to an equivalence on the Karoubi envelope of the additive hulls,
$$\mathscr{H}' \cong \Parity_{I}(\Fl,\L).$$
The goal of this section is to adapt these equivalences to the $\mathbb{G}_m$-equivariant setting.
\begin{thm} \label{firsteq}
We have an equivalence of $\L$-linear graded monoidal categories,
$$\Delta_{\BSvar}: \mathscr{H}_{\BSvar} \cong \Parity_{I \rtimes \mathbb{G}_m}^{\BSvar}(\Fl,\L),$$
lifting to an equivalence
$$\Delta: \mathscr{H} \cong \Parity_{I \rtimes \mathbb{G}_m}(\Fl,\L).$$
\end{thm}
\begin{proof}
The proof is almost identical to that given in \cite[\textsection \textsection 10.3--10.6]{rw} for \cite[Thm. 10.6]{rw}, so for efficiency we merely annotate the meaningful points of difference in specific sections of that paper.

In modifying the proof of \cite{rw}, we must universally replace the ind-varieties $\mathscr{X}$ and $\mathscr{X}^s$ by $\Fl$ and $\Fl^s$, respectively, and the Borel subgroup $\mathscr{B}$ by $I \rtimes \mathbb{G}_m$ (or by $I$, when working with Bott--Samelson resolutions); compare with \cite[Thm. 10.16]{rw} and the remarks preceding it. Recall that we are working in the {\'e}tale context \cite[\textsection 9.3(2)]{rw}, so the ring $\mathbb{Z}'$ should be chosen as $\mathbb{O}$ and then the appropriate analogue of \cite[Lemma 10.3]{rw} holds. No other changes are necessary through to the end of \cite[\textsection 10.3]{rw}; in the subsequent section we have the following:
\begin{itemize}[leftmargin=0.75in]
    \itemb We replace the indented isomorphism with
    \begin{align*}
        \Hom_{D_{I \rtimes \mathbb{G}_m}^b(\Fl,\mathbb{Z}')}(\mathscr{E}(\varnothing),\mathscr{E}(\varnothing)[2m]) & \cong H_{I \rtimes \mathbb{G}_m}^{2m}(\text{pt},\mathbb{Z}') \\ & \cong 
        H_{T \times \mathbb{G}_m}^{2m}(\text{pt},\mathbb{Z}') \\
        & \cong \text{Sym}^m(\mathbb{Z}' \otimes_{\mathbb{Z}} \mathfrak{h}_{\mathbb{Z}}),
    \end{align*}
    where the second-last isomorphism is due to the existence of a surjection $I \rtimes \mathbb{G}_m \to T \times \mathbb{G}_m$ with a unipotent kernel. 
    \vspace{0.25cm}
    \itemc No modification is needed for the description of the image of the upper dot morphism. For the lower dot morphism, note that in the classical setting the identification 
    $$\mathscr{E}_{\mathbb{Z}'}(s) = \underline{\mathbb{Z}}'_X[1] = \mathbb{D} \underline{\mathbb{Z}}'_X[-1]$$
    is canonical after a fixed choice of orientation of $\mathbb{C}$, i.e. of $\sqrt{-1} \in \mathbb{C}$. This can be rephrased as the choice of a continuous isomorphism between the groups $\mathbb{Q}/\mathbb{Z}$ and the roots of unity $\mu_\infty \le \mathbb{C}^\times$. In the {\'e}tale context on the 1-dimensional $\mathbb{F}$-variety $\mathbb{P}^1$, we replace this by a fixed choice of an isomorphism between $\mathbb{Z}_\ell$ and $H^2(\mathbb{P}^1,\mathbb{Z}_\ell)$; see \cite[\textsection 7.1, 7.4]{dan}. This base changes to an isomorphism between $\mathbb{O}$ and $H^2(\mathbb{P}^1,\mathbb{O})$.
    \vspace{0.25cm}
    \itemd The statement of Lemma 10.7 goes through without modification, since we still have
    $$\BSvar(ss) \cong \mathcal{P}_s/I \times \mathcal{P}_s/I \cong \mathbb{P}^1 \times \mathbb{P}^1.$$
    Modulo the selection of adjunction morphisms $a_*, a_!$ (in 10.4.3), Lemma 10.8 and the remainder of 10.4.4 are formal consequences.
    \vspace{0.25cm}
    \iteme For the statement of Lemma 10.9, we must move from $\mathbb{Q}$ to $\mathbb{K}$, and from $\mathscr{B}$-equivariant derived categories on $\mathscr{X}$ to $I \rtimes \mathbb{G}_m$-equivariant categories on $\Fl$. The former change is handled by our analogue of Lemma 10.3. For the latter change, recall that the $I$-equivariant setting on $\Fl$ is already verified by \cite[\textsection 10.7]{rw}, so it remains to replace $I$ with $I \rtimes \mathbb{G}_m$. For this, observe that by \cite[Lemma 2.2]{mr},
    \begin{align*}
    \Rl \otimes_{\Rl'} & \Hom_{D_I^b(\Fl,\mathbb{K})}^\bullet(\mathbb{K}(\mathcal{F}_s),\mathbb{K}(\mathcal{F}_t)) \\
    &\cong \Hom_{D_{I \rtimes \mathbb{G}_m}^b(\Fl,\mathbb{K})}^\bullet(\mathbb{K}(\mathcal{F}_s),\mathbb{K}(\mathcal{F}_t)).
    \end{align*}
    Since $\Rl = \Rl'[d]$ is a polynomial ring over $\Rl'$ and $$\Hom_{D_I^b(\Fl,\mathbb{K})}(\mathbb{K}(\mathcal{F}_s),\mathbb{K}(\mathcal{F}_t)) = \mathbb{K},$$ the same is true for $\Hom_{D_{I \rtimes \mathbb{G}_m}^b(\Fl,\mathbb{K})}(\mathbb{K}(\mathcal{F}_s),\mathbb{K}(\mathcal{F}_t))$.
    
    Lemma 10.10 hinges on the assertion that $\nu_{\underline{w}}$ is birational with connected fibers, as proven in \cite[Lemme 32]{mat}. The birationality is clear, considering the open Schubert cell $IwI/I \subseteq \Fl_w$. Concerning connectedness, note that by definition the target of $\nu_{\underline{w}}$ is a normal variety, so we can apply Zariski's main theorem. 
\end{itemize}
Besides the universal changes indicated above, there are no meaningful alterations to note for \cite[\textsection \textsection 10.5-10.6]{rw} (or the proofs and results those sections reference from earlier in the paper, such as \cite[Proposition 9.17]{rw}).
\end{proof}

\subsection{Second equivalence}
By definition, any object in $\Parity_{I \rtimes \Gm}(\Fl,\bk)$ arises as a direct sum of graded shifts of summands of objects obtained from $\mathscr{E}(\varnothing)$ by convolution with various $\mathscr{E}(s)$. Now $\mathsf{Av}$ commutes with the convolution products discussed in \textsection \ref{s:conv}, i.e.
$\mathsf{Av}(E \star F) = \mathsf{Av}(E) \star F,$ 
and $$\mathsf{Av}(\mathscr{E}(1)) = \Delta_{\IW,\Gm}(1) \cong \nabla_{\IW,\Gm}(1)$$ is a parity object, so if $\mathscr{E} \in \Parity_{I \rtimes \mathbb{G}_m}(\Fl,\bk)$ then $\mathsf{Av}(\mathscr{E}) \in \Parity_{\IW, \mathbb{G}_m}(\Fl,\bk)$. These same observations are made in the proof of \cite[Corollary 11.5]{rw}, and they enable the next definition.

\begin{defn}
Let $\Parity_{\IW,\mathbb{G}_m}^{\BSvar}(\Fl,\bk)$ denote the essential image of $$\mathsf{Av}: \Parity_{I \rtimes \mathbb{G}_m}^{\BSvar}(\Fl,\bk) \to \Parity_{\IW, \mathbb{G}_m}(\Fl,\bk).$$
\end{defn}

In \cite[\textsection 11]{rw}, the equivalences from \textsection \ref{s:first} are extended to diagrams of categories with horizontal equivalences, commuting up to natural isomorphism:
$$
\begin{tikzcd}
\mathscr{H}_{\BSvar}' \arrow{d}{q} \arrow{r}{\Delta_{\BSvar}'} & \Parity_{I}^{\BSvar}(\Fl,\bk) \arrow{d}{\mathsf{Av}} \\
\overline{\mathscr{H}_{\BSvar}'} \arrow{r}{\overline{\Delta_{\BSvar}'}} & \Parity_{\IW}^{\BSvar}(\Fl,\bk).
\end{tikzcd}
\quad
\begin{tikzpicture}
\draw[-stealth,decorate,decoration={snake,amplitude=3pt,pre length=2pt,post length=3pt}] (0,-0.5) -- node[align=center, above] {} ++(1,0);
\end{tikzpicture} \quad
\begin{tikzcd}
\mathscr{H}' \arrow{d}{q} \arrow{r}{\Delta'} & \Parity_{I}(\Fl,\bk) \arrow{d}{\mathsf{Av}} \\
\overline{\mathscr{H}'} \arrow{r}{\overline{\Delta'}} & \Parity_{\IW}(\Fl,\bk).
\end{tikzcd}
$$
Here the overlines denote antispherical quotients \cite[\textsection 1.3]{rw}, while the snake arrow denotes passage to Karoubi envelopes of additive hulls. We now prove the obvious $\mathbb{G}_m$-equivariant analogue.

\begin{thm} \label{secequiv}
We have diagrams of categories with horizontal equivalences, commuting up to natural isomorphism:
$$
\begin{tikzcd}
\mathscr{H}_{\BSvar} \arrow{d}{q} \arrow{r}{\Delta_{\BSvar}} & \Parity_{I \rtimes \mathbb{G}_m}^{\BSvar}(\Fl,\bk) \arrow{d}{\mathsf{Av}} \\
\overline{\mathscr{H}_{\BSvar}} \arrow{r}{\overline{\Delta_{\BSvar}}} & \Parity_{\IW,\mathbb{G}_m}^{\BSvar}(\Fl,\bk).
\end{tikzcd}
\quad
\begin{tikzpicture}
\draw[-stealth,decorate,decoration={snake,amplitude=3pt,pre length=2pt,post length=3pt}] (0,-0.5) -- node[align=center, above] {} ++(1,0);
\end{tikzpicture} \quad
\begin{tikzcd}
\mathscr{H} \arrow{d}{q} \arrow{r}{\Delta} & \Parity_{I \rtimes \mathbb{G}_m}(\Fl,\bk) \arrow{d}{\mathsf{Av}} \\
\overline{\mathscr{H}} \arrow{r}{\overline{\Delta}} & \Parity_{\IW,\mathbb{G}_m}(\Fl,\bk).
\end{tikzcd}
$$
\end{thm}
\begin{proof}
The argument is entirely analogous to the proof of \cite[Theorem 11.11]{rw}, with modifications and annotations as follows. We first make the same universal notational changes as in the proof of Theorem \ref{firsteq}, along with the specialisation $J = S_{\text{f}}$. The next step is to verify the analogue of \cite[Lemma 11.7]{rw}, namely that if $w \in W - {^{\text{f}}W}$ then
\begin{equation} \label{avvanish}
\mathsf{Av}(\mathscr{E}(w)) = 0.
\end{equation}
For this, note that we have a square of categories commuting up to natural isomorphism,
$$
\begin{tikzcd}
\Parity_{I \rtimes \mathbb{G}_m}(\Fl,\bk) \arrow{d}{\mathsf{Av}} \arrow{r}{\For} & \Parity_{I}(\Fl,\bk) \arrow{d}{\mathsf{Av}} \\
\Parity_{\IW, \mathbb{G}_m}(\Fl,\bk) \arrow{r}{\For} & \Parity_{\IW}(\Fl,\bk);
\end{tikzcd}
$$
this is evident from the definition of Av as in \eqref{avg}. Since $\Parity_{\IW, \mathbb{G}_m}(\Fl,\bk)$ is Krull--Schmidt and $\For(\mathscr{E}_{\IW,\mathbb{G}_m}(u)) = \mathscr{E}_{\IW}(u)$ for $u \in {^{\text{f}}}W$ by the general theory of parity sheaves \cite[Lemma 2.4]{mr}, we can conclude that $\For(\mathscr{F}) = 0$ forces $\mathscr{F} = 0$ for $\mathscr{F} \in \Parity_{I \rtimes \mathbb{G}_m}(\Fl,\bk)$. But indeed
$$\For(\mathsf{Av}(\mathscr{E}(w))) = \mathsf{Av}(\For(\mathscr{E}(w))) = 0$$
by \cite[Lemma 11.7]{rw}, so \eqref{avvanish} follows and in fact $\mathsf{Av}(\mathscr{E}(\underline{w})) = 0$ for any reduced expression $\underline{w}$ of $w$. This implies the existence of the functor $\overline{\Delta_{\BSvar}}$. The proof that it is fully faithful proceeds just as in \cite{rw}, with only notational alterations, because \cite[Lemmata 11.1--2]{rw}, the theory of \cite[\textsection 11.3]{rw}, and \cite[Proposition 11.9]{rw} are all immediately adapted. 
\end{proof}

\begin{cor} \label{gract}
The category $\Parity_{\IW,\mathbb{G}_m}(\Fl,\bk)$ admits an action of $\mathscr{H}$ by graded functors, such that $\overline{\Delta}$ is an equivalence of graded right $\mathscr{H}$-modules.
\end{cor}

\section{Geometric action on the Smith quotient} \label{sec:asq}
In this section, we begin by recalling a number of key results and constructions from \cite{st}, particularly the Smith categories of $\mathbb{F}$-varieties with trivial $\varpi$-action and some associated functors. Leveraging an understanding of the morphism spaces between parity objects in the Iwahori--Whittaker Smith category on $\Fl$, we are then able to define and study an action of $\mathscr{H}$ on the parity Smith quotient.

\subsection{Smith categories}
Suppose $X$ is an $\mathbb{F}$-variety with an action of $\mathbb{G}_m$. Recall from \cite{st} the \textit{equivariant Smith category}, defined as the Verdier quotient
$$\Smith(X^{\varpi},\bk) = D_{\mathbb{G}_m}^b(X^\varpi,\bk)/D_{\mathbb{G}_m}^b(X^\varpi,\bk)_{\text{$\varpi$-perf}},$$
where $D_{\mathbb{G}_m}^b(X^\varpi,\bk)_{\text{$\varpi$-perf}}$ is the full subcategory of objects $\mathscr{F}$ for which $\Res_\varpi^{\mathbb{G}_m}(\mathscr{F})$ has \textit{perfect geometric stalks} in the sense of \cite[\textsection 3.3]{rw}. Our main interest will be in the following variant: if $\mathscr{X}$ is as in \textsection \ref{s:par} and $Y \subseteq \mathscr{X}^\varpi$ is a finite union of orbits of $I_\ell^+ = (I^+)^\varpi$, we can consider a category $D_{\IW,\mathbb{G}_m}^b(Y,\bk)$, constructed as in \textsection \ref{s:iw} via a restriction of $\chi$ to $(I_u^+)^\varpi$. (In \cite{st}, this modified construction is denoted $D_{\IW_\ell,\mathbb{G}_m}^b(Y,\bk)$, but we will slightly abuse notation and suppress the subscript $\ell$.) There is then an \textit{Iwahori--Whittaker Smith quotient}
$$\mathsf{Q}_Y: D_{\IW,\mathbb{G}_m}^b(Y,\bk) \to \Smith_{\IW}(Y,\bk) = D_{\IW,\mathbb{G}_m}^b(Y,\bk)/D_{\IW,\mathbb{G}_m}^b(Y,\bk)_{\text{$\varpi$-perf}}.$$
\begin{prop} \label{summary}
Assume that $\mathscr{X}$ is as in \textsection \ref{s:par} with a fixed action of $I^+ \rtimes \mathbb{G}_m$. Let $Y, Z \subseteq \mathscr{X}^\varpi$ be locally closed finite unions of $I_\ell^+$-orbits.
\begin{enumerate}
    \item If $f: Y \to Z$ is a quasi-separated morphism between $\mathbb{F}$-varieties, then for $\dagger \in \{ !, * \}$ there exist functors 
    $$f_\dagger^\Smith: \Smith_{\IW}(Y,\bk) \to \Smith_{\IW}(Z,\bk), \quad f^\dagger_\Smith: \Smith_{\IW}(Z,\bk) \to \Smith_{\IW}(Y,\bk)$$
    such that the following diagrams commute:
    $$
\begin{tikzcd}
D_{\IW,\mathbb{G}_m}^b(Z,\bk) \arrow{d}{\mathsf{Q}_Z} \arrow{r}{f^\dagger} & D_{\IW,\mathbb{G}_m}^b(Y,\bk) \arrow{d}{\mathsf{Q}_Y} \\
\Smith_{\IW}(Z,\bk) \arrow{r}{f_\Smith^\dagger} & \Smith_{\IW}(Y,\bk)
\end{tikzcd}, \quad \begin{tikzcd}
D_{\IW,\mathbb{G}_m}^b(Y,\bk) \arrow{d}{\mathsf{Q}_Y} \arrow{r}{f_\dagger} & D_{\IW,\mathbb{G}_m}^b(Z,\bk) \arrow{d}{\mathsf{Q}_Z} \\
\Smith_{\IW}(Y,\bk) \arrow{r}{f^\Smith_\dagger} & \Smith_{\IW}(Z,\bk).
\end{tikzcd}
$$
We have that $(f_\Smith^*,f_*^\Smith)$ and $(f^\Smith_!,f^!_\Smith)$ are adjoint pairs of functors, so in particular if $f$ is a closed embedding then $f_*^\Smith = f_!^\Smith$ is fully faithful.
\item Up to a change from $\zeta$ to $\zeta^{-1}$, Verdier duality $\mathbb{D}_Y$ preserves $D_{\IW,\mathbb{G}_m}^b(Y,\bk)_{\text{$\varpi$-perf}}$ and therefore descends to a functor on $\Smith_{\IW}(Y,\bk)$.
\item If $X \subseteq \mathscr{X}$ is a locally closed finite union of $I^+$-orbits and $i_X: X^\varpi \hookrightarrow X$ is the inclusion, then the cone of the natural morphism $i_X^! \to i_X^*$ is killed by the Smith quotient functor, yielding a \textit{Smith restriction functor}
$$i_X^{!*}: D_{\IW,\mathbb{G}_m}^b(X,\bk) \to \Smith_{\IW}(X^\varpi,\bk).$$

\item Let $X_1 \subseteq X_2 \subseteq \mathscr{X}$ be locally closed finite unions of $I^+$-orbits, with inclusions $j: X_1 \hookrightarrow X_2$ and $j^\varpi: X_1^\varpi \hookrightarrow X_2^\varpi$. For $\dagger \in \{ !, * \}$, we have
$$i_{X_2}^{!*} \circ j_\dagger \cong (j^\varpi)_\dagger^{\Smith} \circ i_{X_1}^{!*}, \quad i_{X_1}^{!*} \circ j^\dagger \cong (j^\varpi)^\dagger_{\Smith} \circ i_{X_2}^{!*}.$$

\item There is a canonical natural isomorphism,
$$e_Y: \id \cong [2]: \Smith_{\IW}(Y,\bk) \to \Smith_{\IW}(Y,\bk).$$
\end{enumerate}
\end{prop}

\begin{proof}
Most of these statements are proven in \cite[\textsection 6]{st}; the exceptions are (1), which is a generalised version of \cite[Lemma 6.1]{st}, and (2). It is evident from the proof of the latter lemma that our claim holds for $(f^*,f_*)$, so it will suffice to prove (2). Firstly, we have that
$$\Av_\zeta(\mathbb{D} F) \cong \mathbb{D} \Av_{\zeta^{-1}}(F),$$
using an obvious notation to keep track of the $p$-th root of unity chosen for the construction; this follows from the discussion preceding \cite[Lemma 3.8]{bgmrr} and shows that $\mathbb{D}$ respects Iwahori--Whittaker sheaves in the required sense. It remains to prove that $\mathbb{D}(F)$ lies in $D_{\Gm}^b(Y,\bk)_{\text{$\varpi$-perf}}$ if $F$ does. In view of the identification
$$D_\varpi^b(Y,\bk) \cong D^b(Y,\bk[\varpi])$$
and by a standard \textit{d{\'e}visage} argument, we reduce to proving that if $L = j_! \mathcal{L}$ is the extension by zero of a locally constant sheaf on an open stratum $j: U \hookrightarrow Y$, with stalks which are free $\bk[\varpi]$-modules, then the stalks of $\mathbb{D} L$ are likewise. Now, $\mathbb{D} \mathcal{L}$ is a shift and Tate twist of the dual local system $\mathcal{L}^\vee$ over $\bk[\varpi]$, so in particular its stalks are free $\bk[\varpi]$-modules. This shows that $\mathbb{D} \mathcal{L}$ lies in $D_{\mathbb{G}_m}^b(U,\bk)_{\text{$\varpi$-perf}}$. Then 
$\mathbb{D} L = j_* \mathbb{D} \mathcal{L}$ lies in $D_{\mathbb{G}_m}^b(Y,\bk)_{\text{$\varpi$-perf}}$ by \cite[Lemma 3.6]{st}.
\end{proof}

Suppose $\mathscr{U} \subseteq \mathscr{X}^\varpi$ is an ind-scheme, the direct limit of closed finite unions of $I_\ell^+$-orbits. Using Proposition \ref{summary}(1), we define $\Smith_{\IW}(\mathscr{U},\bk)$ to be the direct limit of the categories $\Smith_{\IW}(Y,\bk)$, for $Y \subseteq \mathscr{U}$ a closed finite union of $I_\ell^+$-orbits. Smith functors and the natural isomorphism in Prop. \ref{summary}(5) similarly extend to $\mathscr{U}$.

Using Proposition \ref{summary}(1), we construct functors 
$$(q^s)_\Smith^*: \Smith_{\IW}(\Fl_\ell^s,\bk) \to \Smith_{\IW}(\Fl_\ell,\bk), \quad (q^s)_*^\Smith: \Smith_{\IW}(\Fl_\ell,\bk) \to \Smith_{\IW}(\Fl_\ell^s,\bk),$$
the direct limits of the Smith functors associated to projections between compatible closed finite unions of $I_\ell^+$-orbits on $\Fl_\ell$ and $\Fl_\ell^s$. 

With $\mathscr{X}$ as in \textsection 2.12 and $\alpha \in A_+$, let
$$\mathscr{L}_\Smith^\mathscr{X}(\alpha) = i_{\mathscr{X}_\alpha^+}^{!*}(\mathscr{L}_{\text{AS}}^\mathscr{X}(\alpha)) \in \Smith_{\IW}(\mathscr{X}_\alpha,\bk).$$
We say the orbit $\mathscr{X}_\alpha^+$ is \textit{nice} if
$(\mathscr{X}_\alpha^+)^\varpi$ is an orbit of both $I_\ell^+$ and $I_{u,\ell}^+$. The next result is obtained by adapting the proofs of \cite[Lemma 6.3]{st} and \cite[Proposition 6.5]{st}.
\begin{prop}
Suppose $\mathscr{X}_\alpha^+$ is nice. Then 
$$\Hom_{\Smith_{\IW}((\mathscr{X}_\alpha^+)^\varpi,\bk)}(\mathscr{L}_\Smith^\mathscr{X}(\alpha),\mathscr{L}_\Smith^\mathscr{X}(\alpha)[n]) = \begin{cases}
\bk & \text{if $n$ is even,} \\
0 & \text{if $n$ is odd.}
\end{cases}$$
More generally, if $Y \subseteq \mathscr{X}^\varpi$ is a locally closed finite union of nice orbits, then $\Smith_{\IW}(Y,\bk)$ has finite-dimensional Hom spaces.
\end{prop}

\begin{lem} \label{comm}
Let $f: Y \to Z$ be a quasi-separated morphism between $\mathbb{F}$-varieties as in Prop. \ref{summary}. Then
$$f_\Smith^* e_Z = e_Y f_\Smith^*.$$
\end{lem}
\begin{proof}
By the construction in \cite[Lemma 6.4]{st}, there is $c: \bk_{\pt} \to \bk_{\pt}[2]$ in $D_{\Gm,c}^b(\pt,\bk)$ such that 
$$e_Y \mathsf{Q}_Y = \mathsf{Q}_Y((-) \otimes p_Y^*(c)), \quad e_Z \mathsf{Q}_Z = \mathsf{Q}_Z((-) \otimes p_Z^*(c)),$$
where $p_Y$ and $p_Z$ are the projections to $\pt$. Since $p_Y = p_Z \circ f$, we have
\begin{align*}
f_\Smith^* e_Z \mathsf{Q}_Z = f_\Smith^* \mathsf{Q}_Z((-) \otimes p_Z^*(c)) &= \mathsf{Q}_Y f^* ((-) \otimes p_Z^*(c)) \\
&= \mathsf{Q}_Y (f^*(-) \otimes p_Y^*(c)),
\end{align*}
which is simply $e_Y \mathsf{Q}_Y f^* = e_Y f_\Smith^* \mathsf{Q}_Z$. Using the universal property of the quotient functor $\mathsf{Q}_Z$, we then deduce the claimed equality.
\end{proof}

\subsection{Parity Smith categories}
We are now ready to introduce the theory of parity objects in Smith categories. 
\begin{defn} \label{smpardef}
Suppose $Y \subseteq \mathscr{X}^\varpi$ is a locally closed finite union of $I_\ell^+$-orbits and let $\dagger \in \{ !, * \}$. Consider an ind-scheme $\mathscr{U} \subseteq \mathscr{X}^\varpi$ as above. 
\begin{enumerate}
    \item We say $\mathscr{F} \in \Smith_{\IW}(Y,\bk)$ is \textit{$\dagger$-even}, respectively \textit{$\dagger$-odd},  if for any inclusion
$$j_\alpha^\varpi: \mathscr{X}_\alpha^\varpi \hookrightarrow Y, \quad \alpha \in A_+,$$
the object $(j_\alpha^\varpi)_\Smith^\dagger \mathscr{F}$ decomposes as a direct sum of copies of $\mathscr{L}_\Smith^\mathscr{X}(\alpha)$, respectively $\mathscr{L}_\Smith^\mathscr{X}(\alpha)[1]$. An object is \textit{even}, respectively \textit{odd}, if it is both $!$-even and $*$-even, respectively $!$-odd and $*$-odd.

\item We denote the full subcategory consisting of even objects by $\Smith_{\IW}^0(Y,\bk)$, respectively odd objects by $\Smith_{\IW}^1(Y,\bk)$.
The \textit{parity objects} are defined to lie in the full subcategory $$\Smith_{\IW}^{\parity}(Y,\bk) = \Smith_{\IW}^0(Y,\bk) \oplus \Smith_{\IW}^1(Y,\bk).$$
Taking direct limits over closed $Y \subseteq \mathscr{U}$, we obtain $\Smith_\IW^0(\mathscr{U},\bk)$, $\Smith_\IW^1(\mathscr{U},\bk)$, and $\Smith_\IW^{\parity}(\mathscr{U},\bk)$.
\end{enumerate}
\end{defn}

The next proposition records some of the main features of parity Smith categories established in \cite[\textsection 7]{st}. Note that for $\mathscr{F} \in \Smith_\IW(Y,\bk)$, we take its \textit{support} $\supp \mathscr{F}$ to be the closure of 
$$\bigcup \{ (\mathscr{X}_\alpha^+)^\varpi \subseteq Y: \text{$(j_\alpha^\varpi)_\Smith^* \mathscr{F} \ne 0$ or $(j_\alpha^\varpi)_\Smith^! \mathscr{F} \ne 0$} \}.$$

\begin{prop} \label{nice}
Assume that $\mathscr{X}$ has nice orbits and that $Y \subseteq \mathscr{X}^\varpi$ is a locally closed finite union of orbits.
\begin{enumerate}
    \item All of the Smith categories mentioned in Definition \ref{smpardef}(2) are Krull--Schmidt.
    \item If $\mathscr{F} \in \Smith_{\IW}(Y,\bk)$ is $*$-even and $\mathscr{G} \in \Smith_{\IW}(Y,\bk)$ is $!$-odd, then 
    $$\Hom_{\Smith_{\IW}(Y,\bk)}(\mathscr{F},\mathscr{G}) = 0.$$
     \item If $Z \subseteq Y$ is an open union of $I_\ell^+$-orbits, then indecomposable parity Smith objects on $Y$ are either indecomposable or zero upon restriction to $Z$.
     \item If $\mathscr{F} \in \Smith_{\IW}^{\parity}(Y,\bk)$ is indecomposable, there is exactly one $\alpha \in A_+$ such that
     $(\mathscr{X}_\alpha^+)^\varpi$ is open in $\supp \mathscr{F}$. Conversely, given $\alpha \in A_+$, there is up to isomorphism at most one indecomposable even object $\mathscr{F}$ (resp. odd object $\mathscr{F}[1]$) in $\Smith_{\IW}^{\parity}(Y,\bk)$ containing $(\mathscr{X}_\alpha^+)^\varpi$ as an open subset of its support and restricting to $\mathscr{L}_\Smith^\mathscr{X}(\alpha)$ (resp. to $\mathscr{L}_\Smith^\mathscr{X}(\alpha)[1]$).
\end{enumerate}
\end{prop}

\subsection{Constructing the action} \label{s:cons}
We need to recall two more results in preparation, for which we specialise to the case $\mathscr{X} = \Gr$.

Firstly, recall that if $\kappa, \nu \in \textbf{X}_{+\!+}^\vee$ are such that the $I^+$-orbits $\Gr_{\kappa}^+$, $\Gr_\nu^+$ belong to the same connected component of $\Gr$, then their dimensions $d_\kappa$, $d_\nu$ are of the same parity, which we denote $$\natural(\kappa) = \natural(\nu) \in \mathbb{F}_2.$$ Following \cite[\textsection 7.3]{st}, we define $\Smith^\natural_{\IW}(\Gr^\varpi,\bk)$ to be the full subcategory of $\Smith_{\IW}(\Gr^\varpi,\bk)$ generated by objects whose restriction to $\Gr_{(\nu)}$ is even, respectively odd, if $\natural(\nu) = 0$, respectively $\natural(\nu) = 1$. We then have \cite[Theorem 7.4]{st} a diagram of categories, commuting up to natural isomorphism,
\begin{equation} \label{crux}
\begin{tikzcd}
\Perv_{\IW}(\Gr,\bk) \arrow{r}{\cong} & \Perv_{\IW,\mathbb{G}_m}(\Gr,\bk) \arrow{r}{i_{\Gr}^{!*}} & \Smith_{\IW}(\Gr^\varpi,\bk) \\
\Tilt_{\IW}(\Gr,\bk) \arrow{r}{\cong} \arrow[hookrightarrow]{u} & \Tilt_{\IW,\mathbb{G}_m}(\Gr,\bk) \arrow[hookrightarrow]{u} \arrow{r}{\cong} & \Smith_{\IW}^\natural(\Gr^\varpi,\bk). \arrow[hookrightarrow]{u}
\end{tikzcd}
\end{equation}
The equivalences in the first square are the inverse of $\For_{\mathbb{G}_m}$ and its restriction; we denote the composite equivalence along the bottom row by $i^\natural$. The existence of this diagram, and particularly the equivalence $i^\natural$, is a central and ``miraculous'' result in \cite{st}. It implies the existence of the (even or odd) indecomposable objects 
$$\mathscr{E}_{\Smith}^{\Gr}(\nu) = i^{!*}(\mathscr{E}^\Gr_{\IW,\mathbb{G}_m}(\nu)) \in \Smith^\natural(\Gr^\varpi,\bk), \quad \nu \in \textbf{X}_{+\!+}^\vee.$$

Secondly, as stated in \cite[\textsection 7.4]{st}, the functor $\mathsf{Q}_{\Gr}$ preserves parity objects. The proof of \cite[Proposition 7.6]{st} shows that if $\mathscr{E}, \mathscr{F} \in \Parity_{\IW,\mathbb{G}_m}(\Gr^\varpi,\bk)$ have the same parity, then there are canonical isomorphisms
$$\Hom_{\Smith_{\IW}(\Gr^\varpi,\bk)}(\mathsf{Q}_{\Gr}(\mathscr{E}),\mathsf{Q}_{\Gr}(\mathscr{F})) \cong \bk' \otimes_{H_{\mathbb{G}_m}^\bullet(\pt,\bk)} \Hom_{D_{\IW,\mathbb{G}_m}^b(\Gr^\varpi,\bk)}^\bullet(\mathscr{E}, \mathscr{F})$$
compatible with composition of morphisms in each category, where $\bk'$ denotes $\bk$ viewed as a $H_{\mathbb{G}_m}^\bullet(\pt,\bk)$-module under the map 
$$H_{\mathbb{G}_m}^\bullet(\pt,\bk) \cong \bk[x] \mapsto \bk, \quad x \mapsto 1.$$
Moreover, $\mathsf{Q}_{\Gr}$ is shown to preserve indecomposability of parity objects. Since the inclusions $$j_{(\lambda)}: \Gr_{(\lambda)} \hookrightarrow \Gr, \quad \lambda \in -\overline{\textbf{a}_\ell} \cap \textbf{X}^\vee,$$ induce fully faithful pushforward functors, for every such $\lambda$ and every $\mathscr{E}, \mathscr{F}$ in $\Parity_{\IW,\mathbb{G}_m}(\Gr_{(\lambda)},\bk)$ of the same parity, there are canonical isomorphisms
\begin{equation} \label{quothom}
\Hom_{\Smith_{\IW}(\Gr_{(\lambda)},\bk)}(\mathsf{Q}_{\Gr_{(\lambda)}}(\mathscr{E}),\mathsf{Q}_{\Gr_{(\lambda)}}(\mathscr{F})) \cong \bk' \otimes_{H_{\mathbb{G}_m}^\bullet(\pt,\bk)} \Hom^\bullet(\mathscr{E}, \mathscr{F}).
\end{equation}
Recall from \textsection \ref{s:flag} that for $\lambda = \lambda_0 + \rho^\vee$ and $\mu = \mu_s + \rho^\vee$, we have 
\begin{equation} \label{ident}
\Gr_{(\lambda)} = \Fl_\ell \cong \Fl, \quad \Gr_{(\mu)} = \Fl_\ell^s \cong \Fl^s;
\end{equation}
we fix these identifications for what follows.
\begin{lem} \label{parfact}
The functor $\mathsf{Q} = \mathsf{Q}_{\Fl}: D_{\IW,\mathbb{G}_m}^b(\Fl,\bk) \to \Smith_{\IW}(\Fl,\bk)$ restricts to a functor
$$\Parity_{\IW,\mathbb{G}_m}(\Fl,\bk) \to \Smith_{\IW}^\parity(\Fl,\bk),$$
preserving the parity of even and odd objects.
\end{lem}

\begin{proof}
Up to shift, an indecomposable object in $\Parity_{\IW,\mathbb{G}_m}(\Fl,\bk)$ has the form $\mathscr{F} = \mathscr{E}_{\IW,\mathbb{G}_m}(w)$, where $w \in {^\text{f}}W$. Here the $I^+$-orbit $\Fl_x$ associated to $x \in {^\text{f}}W$ supports a nonzero Iwahori--Whittaker local system, and corresponds to the $I_\ell^+$-orbit $(\Gr_{x \square_\ell \lambda}^+)^\varpi$ under \eqref{ident}. Now
$$(j_x^\varpi)_\Smith^\dagger \mathsf{Q}\mathscr{F} = \mathsf{Q}(j_x^\varpi)^\dagger \mathscr{F},$$
where by assumption $(j_x^\varpi)^\dagger \mathscr{F}$ is a direct sum of graded shifts of copies of $\mathscr{L}_{\text{AS}}^{\Fl}(x)$. In view of the relevant uniqueness statement, we have that $$\mathscr{L}_{\text{AS}}^{\Fl}(x) \cong i_{x \square_\ell \lambda}^* \mathscr{L}_{\text{AS}}^{\Gr}(x \square_\ell \lambda),$$
so $\mathsf{Q} (j_x^\varpi)^\dagger \mathscr{F}$ is a direct sum of graded shifts of copies of $\mathscr{L}_{\Smith}^\Gr(x \square_\ell \lambda)$. Thus we see $\mathsf{Q}$ restricts as described, preserving evenness and oddness. 
\end{proof}

\begin{prop}
The category $\Smith_{\IW}^\parity(\Fl,\bk)$ admits a graded right action of $\mathscr{H}$, such that $\mathsf{Q} = \mathsf{Q}_{\Fl}$ is a graded right $\mathscr{H}$-module functor.
\end{prop}
\begin{proof}
Let $\widetilde{\Parity}(\Fl,\bk)$ be the category whose objects are those of $\Parity_{\IW,\mathbb{G}_m}(\Fl,\bk)$, with Hom spaces
$$\Hom_{\widetilde{\Parity}(\Fl,\bk)}(\mathscr{E},\mathscr{F}) = \bk' \otimes_{H_{\mathbb{G}_m}^\bullet(\pt,\bk)} \bigoplus_n \Hom_{\Parity_{\IW,\mathbb{G}_m}(\Fl,\bk)}^{2n}(\mathscr{E}, \mathscr{F}).$$
Note $\widetilde{\Parity}(\Fl,\bk)$ is naturally equipped with an autoequivalence $[1]$ whose square is naturally isomorphic to the identity functor. Combining Proposition \ref{nice}(1) and (2), \eqref{quothom}, and Lemma \ref{parfact}, we see that $\mathsf{Q}$ factors on the parity subcategory as follows:
\begin{equation} \label{tri}
\begin{tikzcd}
\Parity_{\IW,\mathbb{G}_m}(\Fl,\bk) \arrow{r}{\mathsf{P}} \arrow{dr}[swap]{\mathsf{Q}} & \widetilde{\Parity}(\Fl,\bk) \arrow{d}{\widetilde{\mathsf{Q}}} \\
& \Smith_{\IW}^\parity(\Fl,\bk). 
\end{tikzcd}
\end{equation}
Here $\mathsf{P}$ is trivial on objects and maps morphisms $f$ to simple tensors $1 \otimes f$, while $\widetilde{\mathsf{Q}}$ is essentially surjective and fully faithful, so an equivalence of (graded) categories. By Corollary \ref{gract}, there is a graded right action of $\mathscr{H}$ on $\Parity_{\IW,\mathbb{G}_m}(\Fl,\bk)$ descending to $\widetilde{\Parity}(\Fl,\bk)$; this implies the claim, by transport of structure along $\widetilde{\mathsf{Q}}$.
\end{proof}

Some notation from Prop. \ref{summary}(4): let $e = e_{\Fl}$ and $e^s = e_{\Fl^s}$, and let $e^n$ and $e^{s,n}$ denote the natural isomorphisms $\text{id} \cong [2n]$ induced by $e$ and $e^s$ for $n \ge 1$. For future reference, let us note that if $\varphi: \mathscr{E} \to \mathscr{F}[2n]$ is a morphism between parity objects, then
$$\widetilde{Q}(1 \otimes \varphi) = (e_{\mathsf{Q}\mathscr{F}}^n)^{-1} \mathsf{Q}(\varphi).$$
The functor $\mathsf{Q}^s = \mathsf{Q}_{\Fl^s}$ also respects parity categories and factors through a similar equivalence $\widetilde{\mathsf{Q}^s}: \widetilde{\Parity}(\Fl^s,\bk) \to \Smith_{\IW}^\parity(\Fl^s,\bk)$.

Our understanding of the $\mathscr{H}$-module structure of $\Smith_{\IW}^\parity(\Fl,\bk)$ will hinge on the action of objects $B_s \in \mathscr{H}$, for $s \in S$.
In view of Corollary \ref{gract}, $\mathsf{Av}$ is a right $\mathscr{H}$-module functor, so that we have
\begin{align}
((-) \cdot B_s) \circ \mathsf{Av} &= \mathsf{Av} \circ ((-) \star \mathscr{E}(s)) \nonumber \\
&= \mathsf{Av} \circ (q^s)^* (q^s)_*[1] \label{pp} \\
&= (q^s)^* (q^s)_*[1] \circ \mathsf{Av}; \label{avcom}
\end{align}
here in \eqref{pp} we rely on \cite[Lemma 9.6]{rw} and in \eqref{avcom} we use that $\mathsf{Av}$ commutes with the functor $(q^s)^* (q^s)_*$ (see the proof of \cite[Corollary 11.5]{rw}). But $\mathsf{Av}$ is a quotient functor, so we conclude that
$$(-) \cdot B_s = (q^s)^* (q^s)_*[1],$$
i.e. that $B_s$ acts on $\Parity_{\IW,\mathbb{G}_m}(\Fl,\bk)$ through the push--pull composite $(q^s)^* (q^s)_*[1]$. 

Now, $(q^s)_*$ and $(q^s)^* \cong (q^s)^![-2]$ respect parity objects, so $(q^s)_*^\Smith$ and $(q^s)_\Smith^*$ restrict to functors between the parity Smith categories associated to $\Fl$ and $\Fl^s$. Since they respect gradings, $(q^s)_*$ and $(q^s)^*$ also induce functors
$$\widetilde{(q^s)_*}: \widetilde{\Parity}(\Fl,\bk) \to \widetilde{\Parity}(\Fl^s,\bk), \quad \widetilde{(q^s)^*}: \widetilde{\Parity}(\Fl^s,\bk) \to \widetilde{\Parity}(\Fl,\bk).$$
\begin{prop} \label{bsact}
The object $B_s \in \mathscr{H}$ acts on $\Smith_{\IW}^\parity(\Fl,\bk)$ by the endofunctor $(q^s)_\Smith^* (q^s)_*^\Smith[1]$.
\end{prop}
\begin{proof}
Given the action's construction, we just need to verify that the following two squares commute up to natural isomorphism:
$$\begin{tikzcd}
\widetilde{\Parity}(\Fl,\bk) \arrow{d}{\widetilde{\mathsf{Q}}} \arrow{r}{\widetilde{(q^s)_*}} & \widetilde{\Parity}(\Fl^s,\bk) \arrow{d}{\widetilde{\mathsf{Q}^s}} \\
\Smith_{\IW}^\parity(\Fl,\bk) \arrow{r}{(q^s)_*^\Smith} & \Smith_{\IW}^\parity(\Fl^s,\bk),
\end{tikzcd} \quad \begin{tikzcd}
\widetilde{\Parity}(\Fl^s,\bk) \arrow{d}{\widetilde{\mathsf{Q}^s}} \arrow{r}{\widetilde{(q^s)^*}} & \widetilde{\Parity}(\Fl,\bk) \arrow{d}{\widetilde{\mathsf{Q}}} \\
\Smith_{\IW}^\parity(\Fl^s,\bk) \arrow{r}{(q^s)^*_\Smith} & \Smith_{\IW}^\parity(\Fl,\bk).
\end{tikzcd}
$$
The adjunction $(q^s)^* \dashv (q^s)_*$ yields adjunctions $\widetilde{(q^s)^*} \dashv \widetilde{(q^s)_*}$ and $(q^s)^*_\Smith \dashv (q^s)_*^\Smith$, so by the uniqueness properties of adjoints it suffices to check the case of $(q^s)^*$, i.e. to check that both ways of traversing the right diagram agree. This is clear on objects, so take a morphism $f \in \Hom_{\widetilde{\Parity}(\Fl,\bk)}(\mathscr{E},\mathscr{F})$, which by linearity we may assume to be a simple tensor of the form
$$f = 1 \otimes \varphi, \quad \text{with } \varphi: \mathscr{E} \to \mathscr{F}[2n].$$
By repeated application of Lemma \ref{comm}, we see that
$$(q^s)_\Smith^* e^{s,n} = e^n (q^s)_\Smith^*, \quad \Leftrightarrow \quad (e^n)^{-1} (q^s)_\Smith^* = (q^s)_\Smith^* (e^{s,n})^{-1}.$$
Then
\begin{align*}
    \widetilde{\mathsf{Q}^s}(\widetilde{(q^s)_*}(f)) = \widetilde{\mathsf{Q}^s}(1 \otimes (q^s)_* \varphi) &= \left ( e_{\mathsf{Q}((q^s)_* \mathscr{F})[2n]}^n \right )^{-1} \mathsf{Q}^s((q^s)_* \varphi) \\ 
    &= \left ( e_{\mathsf{Q}((q^s)_* \mathscr{F})[2n]}^n \right )^{-1} (q^s)_*^\Smith \mathsf{Q}(\varphi) \\
    &= (q^s)_*^\Smith \left ( e_{(\mathsf{Q}\mathscr{F})[2n]}^{s,n}\right )^{-1} \mathsf{Q}(\varphi) = (q^s)_*^\Smith \widetilde{\mathsf{Q}}(f),
\end{align*}
as was required to be shown. 
\end{proof}

\section{Bridge to representation theory} \label{sec:bri}
\subsection{Blocks and their functors}
Our goal is now to investigate how the $\mathscr{H}$-action on $\Smith_{\IW}^{\parity}(\Fl,\bk)$ transfers across the equivalence $i^\natural$ from \eqref{crux} and the Iwahori--Whittaker version of the geometric Satake equivalence.

Recall from the proof of \cite[Theorem 8.5]{st} that there is a decomposition of categories 
\begin{equation} \label{tiltdecomp}
\Tilt_{\IW}(\Gr,\bk) = \bigoplus_{\nu \in -\overline{\textbf{a}}_\ell \cap \textbf{X}^\vee} \Tilt_\IW^\nu,
\end{equation}
where $\Tilt_\IW^\nu$ consists of direct sums of the objects $\mathscr{T}_{\IW}(\kappa)$ for $\kappa \in (W \square_\ell \nu) \cap \textbf{X}_{+\!+}^\vee.$ Given \cite[(5.2)]{st}, we have an isomorphism $\mathscr{T}_{\IW}(\kappa) \cong \mathscr{E}_{\IW}^\Gr(\kappa)$ in $D_{\IW}^b(\Gr,\bk)$, so $$i^\natural(\mathscr{T}_{\IW}(\kappa)) = i^{!*}(\mathscr{E}_{\IW,\mathbb{G}_m}^\Gr(\kappa)) = \mathscr{E}_\Smith^\Gr(\kappa)$$
by \cite[(7.1)]{st}; hence $\Smith_\IW^\natural(\Gr^\varpi,\bk)$ decomposes into blocks $\Smith_\IW^\nu$ consisting of the direct sums of objects $\mathscr{E}_\Smith^\Gr(\kappa)$ for $\kappa \in (W \square_\ell \nu) \cap \textbf{X}_{+\!+}^\vee.$ 

Now, notice that
\begin{equation} \label{smcop}
\Smith_{\IW}^\parity(\Gr_{(\nu)},\bk) = \Smith_\IW^\nu \oplus \Smith_\IW^\nu[1];
\end{equation}
in fact, $\Smith_{\IW}^\nu = \Smith_{\IW}^{\natural(\nu)}(\Gr_{(\nu)},\bk)$ by the discussion in \textsection \ref{s:cons}. Under this identification, there is commutative diagram as displayed, with $j_{(\nu)}: \Gr_{(\nu)} \hookrightarrow \Gr^\varpi$:
\begin{equation} \label{prsq}
\begin{tikzcd}
\Smith_{\IW}^\natural(\Gr^\varpi,\bk) \arrow[hookrightarrow]{r} \arrow{d}{\pr_\nu} & \Smith_{\IW}(\Gr^\varpi,\bk) \arrow{d}{(j_{(\nu)})_\Smith^*} \\
\Smith_{\IW}^\nu \arrow[hookrightarrow]{r} & \Smith_{\IW}(\Gr_{(\nu)},\bk). 
\end{tikzcd}
\end{equation}
This is an immediate consequence of the fact that $\Smith_{\IW}^\natural(\Gr^\varpi,\bk)$ is a Krull--Schmidt category whose indecomposable objects are each supported in a single connected component of $\Gr^\varpi$.

At the same time, the Iwahori--Whittaker version of geometric Satake equivalence sends 
$\mathsf{T}(w \bullet_\ell \nu')$ to $\mathscr{T}_{\IW}(w \square_\ell \nu)$, where $w \in W$, $\nu' \in \overline{C}_\ell$, and $\nu = \nu' + \rho^\vee$; see the proof of \cite[Theorem 8.8]{st}. It follows in particular that, for $s \in S$, $\lambda = \lambda_0 + \rho^\vee$, and $\mu = \mu_s + \rho^\vee$,
$$\Tilt_0(\textbf{G}) \cong \Smith_\IW = \Smith_{\IW}^{\lambda}, \quad \Tilt_s(\textbf{G}) \cong \Smith_{\IW}^s = \Smith_{\IW}^{\mu}.$$
The functors $(q^s)^*$ and $(q^s)_* \cong (q^s)^![-2]$ preserve even objects \cite[Lemma 9.6(2)]{rw}, so Lemma \ref{parfact} implies the existence of an endofunctor
$$(q^s)_\Smith^*(q^s)_*^\Smith: \Smith_{\IW} \to \Smith_{\IW}.$$
By de-grading $\Smith_{\IW}^\parity(\Fl,\bk)$, identifying copies of $\Smith_\IW$ as in \eqref{smcop}, we obtain the following corollary to Proposition \ref{bsact}:

\begin{prop}
There is a right action of $\mathscr{H}$ on $\Smith_\IW$, with $B_s \langle n \rangle$ acting by $(q^s)_\Smith^* (q^s)_*^\Smith$.
\end{prop}

We spend much of the remainder of the paper comparing this endofunctor to the wall-crossing functor $$\theta_s = T_s T^s: \Tilt_0(\textbf{G}) \to \Tilt_s(\textbf{G}) \to \Tilt_0(\textbf{G}).$$

\subsection{Translation functors and fixed points on the affine Grassmannian}
Recall that for $\lambda_0, \mu_s$ as above, there is a unique element $\gamma \in \textbf{X}_+^\vee \cap W_{\text{f}}(\mu_s - \lambda_0),$
and we define 
$$T^s: \Tilt_0(\textbf{G}) \to \Tilt_s(\textbf{G}), \quad M \mapsto \pr_{\mu_s}(M \otimes \mathsf{T}(\gamma)).$$
(We could replace $\mathsf{T}(\gamma)$ by any finite-dimensional module with highest weight $\gamma$ and obtain an isomorphic functor; see \cite[Rmk. 7.6]{jan}.) Remembering that convolution affords an action of $D_{L^+ G}^b(\Gr,\bk)$ on $D_{\IW}^b(\Gr,\bk)$ and $D_{\IW,\Gm}^b(\Gr,\bk)$, we see that the geometric Satake equivalence and its variants yield a diagram of categories (commuting up to natural isomorphism), 
\begin{equation} \label{trans}
\begin{tikzcd}[column sep=large]
\Tilt_0(\textbf{G}) \arrow{d}{T^s} \arrow{r}{\cong} & \Tilt_\IW^0 \arrow{d}{\tau_\IW^s} \arrow{r}{\cong} & \Tilt_{\IW,\Gm}^0 \arrow{d}{\tau_{\IW,\Gm}^s} \\
\Tilt_s(\textbf{G}) \arrow{r}{\cong} & \Tilt_\IW^s \arrow{r}{\cong} & \Tilt_{\IW,\Gm}^s,
\end{tikzcd}
\end{equation}
with $\tau_\IW^s = \pr_{\mu_s}((-) \star \mathscr{T}(\gamma))$ and similarly for $\tau_{\IW,\Gm}^s$. Naturally, then, our interest turns to the functor $(-) \star \mathscr{T}(\gamma)$. To analyse this functor, we need the following technical lemma connecting objects in constructible derived categories with intersection cohomology complexes on their supports. In its proof, we encounter the \textit{recollement} (``gluing'') situation, as explained in detail in \cite{recollement}.

\begin{lem} \label{devisage}
Assume $X$ is a stratified ind-variety,
$$X = \bigsqcup_{\zeta \in \Lambda} X^\zeta,$$
where the strata $X^\zeta$ are locally closed and simply connected varieties, such that $$\overline{X^\zeta} = \bigsqcup_{\zeta' \le \zeta} X^{\zeta'}.$$ Let $\mathfrak{I} \subseteq \Lambda$ be a finite downward-closed partially ordered subset and write $$X^\mathfrak{I} = \bigsqcup_{\zeta \in \mathfrak{I}} X_\zeta.$$ If $\mathcal{F} \in D_\Lambda^b(X)$ is supported on $X^\mathfrak{I}$, then
$$\mathcal{F} \in D(\mathfrak{I}) = \langle \IC(\zeta): \zeta \in \mathfrak{I} \rangle_\Delta,$$
the full triangulated subcategory generated by the $\IC(\zeta)$. This coincides with the full triangulated subcategory generated by the tilting objects $\mathscr{T}(\zeta)$, $\zeta \in \mathfrak{J}$.
\end{lem}
\begin{proof}
In the following, the symbols ${^p} \tau$ and ${^p}\mathscr{H}$ refer to perverse truncation and perverse cohomology functors on the constructible derived category $D_\Lambda^b(X)$, respectively. The second claim follows from the first by \cite[Prop. 7.17]{riche}, so we need only prove the first claim. For this we use induction on $n = |\mathfrak{I}|$. 

If $n = 1$, then $\mathfrak{I} = \{ \zeta \}$ with
$$X^\zeta = \overline{X^\zeta}.$$
Let $i$ be maximal with ${^p}\mathscr{H}^i(\mathcal{F}) \ne 0$, so there is a distinguished triangle in $D_{\Lambda}^b(X)$,
$${^p} \tau_{< i} \mathcal{F} \to \mathcal{F} \to {^p}\mathscr{H}^i(\mathcal{F})[-i] \overset{[1]}{\to}.$$
Now if $i'$ is maximal such that ${^p}\mathscr{H}^{i'}({^p}\tau_{< i} \mathcal{F}) \ne 0$, then by construction $i' < i$; induction on $i$ now settles the case $n = 1$.

Suppose now that $n > 1$ and choose $\zeta \in \mathfrak{I}$ maximal. Then
$$X^\mathfrak{I} - X^\zeta = \bigsqcup_{\zeta' \in \mathfrak{I}} X^{\zeta'} - X^\zeta = \bigsqcup_{\zeta \ne \zeta' \in \mathfrak{I}} X^{\zeta'};$$
since $\overline{X^{\zeta'}}$ is a union of strata $X^{\zeta''}$ with $\zeta'' < \zeta'$, the maximality of $\zeta$ shows that 
$$\bigsqcup_{\zeta \ne \zeta' \in \mathfrak{I}} X^{\zeta'} = \bigcup_{\zeta \ne \zeta' \in \mathfrak{I}} \overline{X^{\zeta'}}$$
is closed (by finiteness of $\mathfrak{I}$). Hence we have a \textit{recollement} situation,
$$X^\zeta \hookrightarrow X^\mathfrak{I} \hookleftarrow X^\mathfrak{I} - X^\zeta;$$
call these inclusions $j$ and $i$, respectively. Now, by assumption, $\mathcal{F}$ is the pushforward of some $\mathcal{F}' \in D_{\Lambda}^b(X^\mathfrak{I})$. Accordingly, we can form a distinguished triangle
$$j_! j^! \mathcal{F}' \to \mathcal{F}' \to i_* i^* \mathcal{F}' \overset{[1]}{\to}.$$
Note that $i_* i^* \mathcal{F}'$ is supported on $X^{\mathfrak{I}'}$, where $\mathfrak{I}' = \{ \zeta' \in \mathfrak{I}: \zeta' < \zeta \}$ is a proper subset of $\mathfrak{I}$, so by induction $i_* i^* \mathcal{F}' \in D(\mathfrak{I})$. We reduce to showing $j_! j^! \mathcal{F}' \in D(\mathfrak{I})$. Note that $j^! \mathcal{F}'$ is a sheaf on $X^\zeta$, so by the $n = 1$ case above (with $X$ replaced by $X^\zeta$ and $\Lambda$ and $\mathfrak{I}$ by $\{ \zeta \}$) we can conclude that
$$j^! \mathcal{F}' \in \langle \underline{\bk}_{X^\zeta}[d_\zeta] \rangle_\Delta,$$
and hence $j_! j^! \mathcal{F}' \in \langle j_! \underline{\bk}_{X^\zeta}[d_\zeta] \rangle_\Delta.$ Thus it will suffice to prove that $j_! \underline{\bk}_{X^\zeta}[d_\zeta] \in D(\mathfrak{I})$. Since $j_!$ is right $t$-exact, $j_! \underline{\bk}_{X^\zeta}[d_\zeta] \in {^p}D^{\le 0}$. Now we have a distinguished triangle,
$${^p} \tau_{< 0}(j_! \underline{\bk}_{X^\zeta}[d_\zeta]) \to j_! \underline{\bk}_{X^\zeta}[d_\zeta] \to {^p}j_! \underline{\bk}_{X^\zeta}[d_\zeta] \overset{[1]}{\to}.$$
Using the identification $j^! \underline{\bk}_{X^\mathfrak{I}} \cong \underline{\bk}_{X^\zeta}$ as well as the adjunctions $(j_!,j^!)$ and $({^p}j_!,j^!)$, we can see that
$$j^* j_! \underline{\bk}_{X^\zeta} = j^! j_! \underline{\bk}_{X^\zeta} = j^! j_! j^! \underline{\bk}_{X^\mathfrak{I}} = j^! \underline{\bk}_{X^\mathfrak{I}} = j^! ({^p}j_!) j^! \underline{\bk}_{X^\mathfrak{I}} = j^* ({^p}j_!) \underline{\bk}_{X^\zeta}.$$ 
The upshot is that ${^p} \tau_{< 0}(j_! \underline{\bk}_{X^\zeta}[d_\zeta])$ is supported on $\mathfrak{I} - \{ \zeta \}$, so by the inductive assumption on $n$ we reduce to considering ${^p}j_! \underline{\bk}_{X^\zeta}[d_\zeta]$. This is a perverse sheaf, for which the result is well known.
\end{proof} 

From the restriction isomorphism $(j^\gamma)^* \mathscr{T}(\gamma) \cong \underline{\bk}_{\Gr^\gamma}[\dim(\Gr^\gamma)]$ and the adjunction isomorphism
$$\Hom((j^\gamma)^* \mathscr{T}(\gamma),\underline{\bk}_{\Gr^\gamma}[\dim(\Gr^\gamma)]) \cong \Hom(\mathscr{T}(\gamma),\mathscr{K}_\gamma),$$ for $\mathscr{K}_\gamma = (j^\gamma)_* \underline{\bk}_{\Gr^\gamma}[\dim(\Gr^\gamma)]$, we produce a distinguished triangle in $D_{L^+G}^b(\Gr,\bk)$,
\begin{equation} \label{dt}
C \to \mathscr{T}(\gamma) \to \mathscr{K}_\gamma \overset{[1]}{\to}.
\end{equation}
Evidently $(j^\gamma)^* C = 0$, so $C$ is supported on the union of the $\Gr^\zeta$ with $\zeta < \gamma$ and by Lemma \ref{devisage} we infer that $C \in \langle \IC(\zeta): \zeta < \gamma \rangle_\Delta$. This control on $C$ will be shown to justify a final shift of attention, to the functor 
$$(-) \star \mathscr{K}_\gamma: D_{\IW}^b(\Gr,\bk) \to D_{\IW}^b(\Gr,\bk).$$

For the remainder of this section, we let 
\begin{align*}
   \mathscr{Y}_*^\gamma=LG \times^{L^+G} \Gr^\gamma & \overset{\iota}{\longhookrightarrow} \mathscr{Y}_* = LG \times^{L^+G} \Gr, \quad \mathscr{Y} = \Gr \times \Gr.
\end{align*}
\begin{prop} \label{conv}
Consider the following diagram:
\begin{center}
\begin{tikzcd}[column sep=normal]
\Gr & \mathscr{Y}_*^\gamma \arrow{l}[above]{\pi} \arrow{r}{m} & \Gr;
\end{tikzcd}
\end{center}
here $\pi$ is a projection and $m$ is induced by multiplication. There is a natural isomorphism,
$$\mathscr{F} \star \mathscr{K}_\gamma \cong m_* \pi^* \mathscr{F}, \quad \mathscr{F} \in D_{\IW}^b(\Gr,\bk).$$
\end{prop}\begin{proof}
Let $j^\gamma: \Gr^\gamma \hookrightarrow \Gr$. We have a commutative diagram,
\begin{center}
\begin{tikzcd}[column sep=normal]
\Gr & \mathscr{Y}_*^\gamma \arrow{l}[above]{\pi} \arrow[hookrightarrow]{r}{\iota} & \mathscr{Y}_* \arrow{r}{m'} & \Gr \\
LG \arrow{u}{p} & LG \times \Gr^\gamma \arrow{l}[above]{\pi_1} \arrow{u}{q} \arrow[hookrightarrow]{r}{1 \times j^\gamma} & LG \times \Gr, \arrow{u}{q'}
\end{tikzcd}
\end{center}
where the vertical arrows are the obvious quotients, $\pi_1$ is the projection onto the first factor, and $m'$ is induced by the usual multiplication map. Note that $m = m' \iota$ and that the right-most square is cartesian. By definition,
$\mathscr{F} \star \mathscr{K}_\gamma = m'_*(\mathscr{F} \widetilde{\boxtimes} \mathscr{K}_\gamma),$
where $\mathscr{F} \widetilde{\boxtimes} \mathscr{K}_\gamma$ is the unique object in $D_{L^+G}^b(LG \times^{L^+G} \Gr)$ for which
$$(q')^*(\mathscr{F} \widetilde{\boxtimes} \mathscr{K}_\gamma) \cong p^* \mathscr{F} \boxtimes \mathscr{K}_\gamma.$$
We claim that $\mathscr{F} \widetilde{\boxtimes} \mathscr{K}_\gamma \cong \iota_* \pi^* \mathscr{F}$. Indeed, using smooth base change,
$$(q')^* \iota_* \pi^* \mathscr{F} \cong (1 \times j^\gamma)_* q^* \pi^* \cong (1 \times j^\gamma)_* (p \circ \pi_1)^* \mathscr{F}.$$
Now $p \circ \pi_1 = p \times \pt$, so $(p \circ \pi_1)^* \mathscr{F} \cong p^* \mathscr{F} \boxtimes \underline{\bk}_{\Gr^\gamma}$ and we have
$$(1 \times j^\gamma)_* (p^* \mathscr{F} \boxtimes \underline{\bk}_{\Gr^\gamma}) \cong p^* \mathscr{F} \boxtimes \mathscr{K}_\gamma.$$
Accordingly, 
$\mathscr{F} \star \mathscr{K}_\gamma = m'_*(\mathscr{F} \widetilde{\boxtimes} \mathscr{K}_\gamma) \cong m'_* \iota_* \pi^* \mathscr{F} = m_* \pi^* \mathscr{F}.$
\end{proof}

It will become convenient to work with an untwisted version of the functor $m_* \pi^*$. As mentioned in the proof of \cite[Lemma 4.4]{mv07}, the product of projection and multiplication yields an isomorphism of ind-schemes $$\pi \times m: \mathscr{Y}_* \xrightarrow{\sim} \mathscr{Y},$$ in the notation of Proposition \ref{conv}; it is $\mathbb{G}_m$-equivariant, since $\pi$ and $m$ are. Denote by $\mathscr{Y}^\gamma \subseteq \mathscr{Y}$ the image of $\pi \times m$ restricted to $\mathscr{Y}_*^\gamma$. Then we
have a commutative diagram
\begin{center}
\begin{tikzcd}[column sep=normal]
 \Gr & \mathscr{Y}_*^\gamma \arrow{d}{\rotatebox{270}{$\sim$}} \arrow{l}[above]{\pi} \arrow{r}{m} & \Gr \\
& \mathscr{Y}^\gamma \arrow{ul}{\phi_1} \arrow{ur}[below]{\phi_2} & 
\end{tikzcd}
\end{center}
where the $\phi_i$ are projection maps, and $\mathscr{Y}^\gamma$ is stable under the natural diagonal action of $I^+$ (and even $L^+G$) on $\Gr \times \Gr$. 

\subsection{Translation functors in Smith theory}
Ultimately, we will show that $(\phi_2)_*\phi_1^*$ induces an action in Smith theory which coincides with $(q^s)_*^\Smith$; the next proposition facilitates the first step towards this goal.

\begin{prop}
The following diagram commutes up to natural isomorphism:
\begin{center}
\begin{tikzcd}
 D_{\IW,\mathbb{G}_m}^b(\Gr)  \arrow{d}{i_{\Gr}^{!*}} \arrow{r}{\phi_1^*} & D_{\IW,\mathbb{G}_m}^b(\mathscr{Y}^\gamma) \arrow{d}{i_{\mathscr{Y}^\gamma}^{!*}} \arrow{r}{(\phi_2)_*} & \arrow{d}{i_{\Gr}^{!*}} D_{\IW,\mathbb{G}_m}^b(\Gr) \\
 \Smith_{\IW}(\Gr^\varpi)  \arrow{r}{(\phi_1^\varpi)_\Smith^*} & \Smith_{\IW}((\mathscr{Y}^\gamma)^\varpi) \arrow{r}{(\phi_2^\varpi)_*^\Smith} &  \Smith_{\IW}(\Gr^\varpi).
\end{tikzcd}
\end{center}
\end{prop}
\begin{proof}
That the left square commutes is merely a consequence of functoriality:
$$(\phi_1^\varpi)_\Smith^* \circ i_\Gr^{!*} = (\phi_1^\varpi)_\Smith^* \mathsf{Q} i_{\Gr}^* = \mathsf{Q} (i_\Gr \circ \phi_1^\varpi)^* = \mathsf{Q} (\phi_1 \circ i_{\mathscr{Y}_0})^* = i_{\mathscr{Y}_0}^{!*} \circ \phi_1^*.$$
For the right square, we establish a more general base change result. Let $f: A \to B$ be a $\varpi$-equivariant quasi-separated morphism, where $A \subseteq \mathscr{Y}^\gamma$ and $B \subseteq \Gr$ are closed finite unions of $I^+$-orbits. Then we claim the following diagram commutes:
\begin{equation} \label{base}
\begin{tikzcd}
 D_{\IW,\mathbb{G}_m}^b(A) \arrow{d}{i_A^{!*}} \arrow{r}{f_*} & \arrow{d}{i_{B}^{!*}} D_{\IW,\mathbb{G}_m}^b(B) \\
\Smith_{\IW}(A^\varpi)\arrow{r}{(f^\varpi)_*^\Smith} &\Smith_{\IW}(B^\varpi).
\end{tikzcd}
\end{equation}
Indeed, suppose $P$ is the pullback in a cartesian square, 
\begin{center}
\begin{tikzcd}
 A \arrow{r}{f} & B \\
P \arrow{r}{f'} \arrow{u}{i'} & B^\varpi. \arrow{u}{i_B}
\end{tikzcd}
\end{center}
The universal $a: A^\varpi \to P$ is a closed immersion, so we have a \textit{recollement} situation
$$A^\varpi \rightarrow P \leftarrow U,$$
where $j: U \hookrightarrow P$ has a free $\varpi$-action. Now $f' \circ j: U \to B^\varpi$ is $\varpi$-equivariant so factors as $$\overline{j} \circ q_U: U \to U/\varpi \to B^\varpi.$$
By the proof of \cite[Prop. 2.6]{st}, we see that $(f' \circ j)_* H$ has perfect geometric stalks if $H$ is an object in $D_{\varpi,c}^b(U)$. After composing with $f'_*$, a distinguished triangle in $D_{\IW,\Gm}^b(P)$ of the form
$$a_! a^! F \to F \to j_* j^* F \overset{[1]}{\to}$$
becomes a distinguished triangle in $D_{\IW,\Gm}^b(B^\varpi),$
$$f'_* a_! a^! F \to f'_* F \to (f' \circ j)_* j^*F \overset{[1]}{\to},$$
where $f'_* a_! a^! F = f_*^\varpi a^! F$. Take Smith quotients, with $H = \text{Res}_\varpi^{\Gm}(j^*F)$ in mind:
$$(f^\varpi)_*^\Smith \mathsf{Q}_A(a^! F) \cong \mathsf{Q}_B(f'_*F).$$
Assume that $F = (i')^!E$, so that $a^! F = i_A^! E$. We are then left with 
$$(f^\varpi)_*^\Smith i_A^{!*} E \cong \mathsf{Q}_B(f'_* i'^! E) \cong \mathsf{Q}_B(i_B^! f_* E) = i_B^{!*} f_* E$$
as desired; hence, \eqref{base} commutes after all. Now the projection $\phi_2$ is approximated by quasi-separated morphisms between closed finite unions of $I^+$-orbits, so we can take a direct limit over fixed points to obtain the result.
\end{proof}

To bring the Smith categories associated with partial affine flag varieties into our calculation, we must work with the pullback ind-schemes $\mathscr{W}$ and $\mathscr{Z}$ as displayed:
\begin{equation} \label{twosq}
\begin{tikzcd}
\mathscr{W} \arrow{d}{w} \arrow{r}{\eta} & (\mathscr{Y}^\gamma)^\varpi \arrow{d}{\phi_2^\varpi} \\
\Fl_\ell^s \arrow[hookrightarrow]{r}{h_\ell^s} & \Gr^\varpi
\end{tikzcd}, \quad 
\begin{tikzcd}
\mathscr{Z} \arrow{d}{z} \arrow{r}{\theta} & \mathscr{W} \arrow{d}{} \\
\Fl_\ell \arrow[hookrightarrow]{r}{h_\ell} & \Gr^\varpi;
\end{tikzcd}
\end{equation}
here $h_\ell$, $h_\ell^s$ correspond to $j_{(\lambda)}$, $j_{(\mu)}$ when identifying $\Fl_\ell \cong \Gr_{(\lambda)}$, $\Fl_\ell^s \cong \Gr_{(\mu)}$.

\begin{prop}
The ind-scheme $\mathscr{Z}$ identifies with the graph of $q^s: \Fl_\ell \to \Fl_\ell^s$.
\end{prop}
\begin{proof}
It can be verified formally that the following square is cartesian:
\begin{center}
\begin{tikzcd}
\mathscr{Z} \arrow{d}{z \times z^s} \arrow{r}{\eta \theta} & (\mathscr{Y}^\gamma)^\varpi \arrow[hookrightarrow]{d}{} \\
\Fl_\ell \times \Fl_\ell^s \arrow[hookrightarrow]{r}{h_\ell\times h_\ell^s} & \Gr^\varpi \times \Gr^\varpi
\end{tikzcd}
\end{center}
where $z^s = w\theta$, or in other words that 
\begin{equation} \label{zexpr}
\mathscr{Z} = (\mathscr{Y}^\gamma)^\varpi \cap \Fl_\ell \times \Fl_\ell^s.
\end{equation}
Now, a typical point of $\mathscr{Y}_*^\gamma$ has the form $[g,L_\gamma]$, since $[h,h'L_\gamma] = [hh',L_\gamma]$ for $h \in LG$, $h' \in L^+G$. Then the points of $\mathscr{Y}^\gamma$ have the form
$$(\pi,m)[g,L_\gamma] = (gL_0, gL_\gamma),$$
or equivalently the form $(gL_\lambda,gL_\mu)$ with $g \in LG$. Accordingly, by \eqref{zexpr}, we see that the points of $\mathscr{Z}$ have the form $(fL_\lambda, hL_\mu)$ for $f, h \in L_\ell G$, where there is some $g \in LG$ such that
$$fL_\lambda = g L_\lambda, \quad h L_\mu = g L_\mu.$$
Phrasing this condition differently, we require that
$$fz^\lambda \alpha z^{-\lambda} = h z^\mu \beta z^{-\mu}$$
for some $\alpha, \beta \in L^+G$ (the common value being $g \in LG$). It follows that the fibre of the first projection $z: \mathscr{Z} \to \Fl_\ell$ over a fixed $f L_\lambda$, for $f \in L_\ell G$, can be identified with the points $h L_\mu \in \Fl_\ell^s$ for $h \in f z^\lambda L^+ G z^{-\lambda} z^\mu L^+ G z^{-\mu} \cap L_\ell G$, or equivalently with the quotient
$$\frac{f H_\lambda H_\mu \cap L_\ell G}{H_\mu \cap L_\ell G} \cong \frac{H_\lambda H_\mu \cap L_\ell G}{H_\mu \cap L_\ell G} = F,$$
writing $H_\nu = z^\nu L^+G z^{-\nu}$ and letting $H_\mu \cap L_\ell G$ act by multiplication on the right. Of course, it will be sufficient to show that the fibre $F = \{ 1 \}$; this is the remaining aim of the proof. 

Notice first that we have $\mathbb{G}_m$-equivariant maps
$$F \hookrightarrow \frac{H_\lambda H_\mu}{H_\lambda} \cong \frac{H_\lambda}{H_\lambda \cap H_\mu} = H_\lambda \cdot L_\mu.$$
Thus $F \hookrightarrow (H_\lambda \cdot L_\mu)^\varpi$. To calculate the latter, recall first \cite[4.6]{st} the decomposition
$$\Gr^\nu = \bigsqcup_{\nu' \in W_{\text{f}}(\nu)} \Gr_{\nu'}, \quad \nu \in \textbf{X}^\vee.$$
It follows that
$$H_\lambda \cdot L_\mu = \bigsqcup_{\nu \in W_{\text{f}}(\mu-\lambda)+\lambda} I_\lambda \cdot L_\nu, \quad I_\lambda = z^\lambda I z^{-\lambda}.$$
For $\alpha \in \Phi$, let $\delta_\alpha \in \mathbb{Z}$ be $1$ or $0$ according to whether $\alpha \in \Phi_+$ or not. Then recall \cite[Lemma 4.8]{st} that we have a $\mathbb{G}_m$-equivariant isomorphism
$$I_u^\nu = \prod_\alpha \left ( \prod_{\delta_\alpha \le m < \langle \alpha, \nu \rangle} U_{\alpha + m \delta} \right ) \cong \Gr_{\nu}, \quad x \mapsto x L_\nu.$$
Accordingly, the same formula defines an equivariant isomorphism
$$I_u^{\lambda,\nu} = \prod_\alpha \left ( \prod_{\delta_\alpha + \langle \alpha, \lambda \rangle \le m < \langle \alpha, \nu \rangle} U_{\alpha + m \delta} \right ) \cong I_\lambda \cdot L_\nu.$$
Hence, taking $\varpi$-fixed points,
$$(I_u^{\lambda,\nu})^{\varpi} = \prod_\alpha \left ( \prod_{\substack{\delta_\alpha + \langle \alpha, \lambda \rangle \le m < \langle \alpha, \nu \rangle \\ \ell \mid m}} U_{\alpha + m \delta} \right ) \cong (I_\lambda \cdot L_\nu)^\varpi.$$
Let $\nu = w(\mu - \lambda) + \lambda$ for $w \in W_{\text{f}}$, and suppose $$u L_\nu = uz^\nu L^+G = u z^{\nu - \mu} L_\mu \in F$$ 
for $u \in (I_u^{\lambda,\nu})^{\varpi} \subseteq L_\ell G$. Then there is $h_\lambda h_\mu \in H_\lambda H_\mu \cap L_\ell G$ with 
$$z^{\mu - \nu} u^{-1} h_\lambda h_\mu \in H_\mu.$$
This element also belongs to $z^{\mu -\nu} L_\ell G$ and hence to the intersection $H_\mu \cap z^{\mu -\nu} L_\ell G$. However, if $\nu \ne \mu$, then we claim $H_\mu \cap z^{\mu-\nu} L_\ell G$ is empty. By translation, it suffices to prove that
$$z^{\nu} L^+ G \cap L_\ell G z^\mu = \varnothing.$$
If we suppose otherwise, then by acting on $L_0 \in \Gr$ we find that $L_{\nu} \in \Fl_\ell^s$. By \cite[Remark 4.9]{st}, this implies that
$\nu \in W \square_\ell \mu,$
and hence that $u(\mu) = w(\mu - \lambda) + \lambda$ for some $u \in W$. Rewriting,
$$u \bullet \mu_s = w(\mu_s - \lambda_0) + \lambda_0.$$
Now, by \cite[Lemma II.7.7]{jan}, this forces $\mu_s = u \bullet \mu_s$, so that $w(\mu - \lambda) = \mu - \lambda$, contradicting the assumption $\nu \ne \mu$. So we reduce to $\nu = \mu$, in which case we recall that, for $\alpha \in \Phi_+$,
$$0 < \langle \alpha, \lambda \rangle < \ell, \quad 0 \le \langle \alpha, \mu \rangle \le \ell;$$
these inequalities imply $(I_u^{\lambda,\mu})^\varpi = 1$ and hence that $F = F \cap (I_\lambda \cdot L_\mu)^\varpi = \{ 1 \}$.
\end{proof}
The upshot of our efforts is a type of ``base change'' in Smith theory along the diagram below:
\begin{center}
\begin{tikzcd} 
 \Gr^\varpi & \arrow{l}[swap]{\phi_1^\varpi} (\mathscr{Y}^\gamma)^\varpi \arrow{r}{\phi_2^\varpi} & \Gr^\varpi \\
 \Fl_\ell \arrow[bend right=30,swap]{rr}{q^s} \arrow{u}{h_\ell} & \arrow{l}[swap]{z} \mathscr{Z} \arrow{u} \arrow{r}{z^s} & \Fl_\ell^s \arrow{u}{h_\ell^s}
 \end{tikzcd}
\end{center}
To be more precise, we can combine the cartesian squares in \eqref{twosq} into one commutative diagram, now featuring the morphism $q^s$.
\begin{equation} \label{big}
\begin{tikzcd} 
 \Gr^\varpi & \arrow{l}[swap]{\phi_1^\varpi} (\mathscr{Y}^\gamma)^\varpi \arrow{r}{\phi_2^\varpi} & \Gr^\varpi \\
 \Gr^\varpi \arrow[equal]{u} & \mathscr{W} \arrow{u}{\eta} \arrow{l}{} \arrow{r}{w} & \Fl_\ell^s \arrow{u}{h_\ell^s} \\
 \Fl_\ell \arrow[bend right=30,swap]{rr}{q^s} \arrow{u}{h_\ell} & \arrow{l}[swap]{z} \mathscr{Z} \arrow{u}{\theta} \arrow{r}{z^s} & \Fl_\ell^s \arrow[equal]{u}.
 \end{tikzcd}
\end{equation}

Let us now perform the main calculation, aiming to show that Smith restriction intertwines translation onto the wall and pushforward in Smith theory:
\begin{center}
\begin{tikzcd}[column sep=large]
\Tilt_{\IW,\Gm}^0 \arrow{d}{\tau_{\IW,\Gm}^s} \arrow{r}{i_0^{!*}} & \Smith_{\IW} \arrow{d}{(q^s)_*^\Smith} \\
\Tilt_{\IW,\Gm}^s \arrow{r}{i_s^{!*}} & \Smith_{\IW}^s
\end{tikzcd}
\end{center}
Begin with any $\mathscr{F} \in \Tilt_{\IW,\Gm}^0$. Then
\begin{align}
    i_s^{!*} \tau_{\IW,\Gm}^s \mathscr{F} = i_s^{!*} \pr_{\mu}(\mathscr{F} \star \mathscr{T}(\gamma)) &= \pr_{\mu} i^{!*} (\mathscr{F} \star \mathscr{T}(\gamma)) = (h_\ell^s)_\Smith^* i^{!*} (\mathscr{F} \star \mathscr{T}(\gamma)), \nonumber
\end{align}
in view of \eqref{prsq}. Now, recall the distinguished triangle \eqref{dt}; convolving on the left by $\mathscr{F}$, we obtain
\begin{equation} \label{dt2}
\mathscr{F} \star C \to \mathscr{F} \star \mathscr{T}(\gamma) \to \mathscr{F} \star \mathscr{K}_\gamma \overset{[1]}{\to},
\end{equation}
where $\mathscr{F} \star C$ belongs to the triangulated category generated by objects $\mathscr{F} \star \mathscr{T}(\zeta)$ for $\zeta < \gamma$, by Lemma \ref{devisage}. It follows from the (equivariant) Iwahori--Whittaker version of the geometric Satake equivalence and \cite[Rmk. 7.7]{jan} that $\pr_\mu(\mathscr{F} \star \mathscr{T}(\zeta)) = 0$ in $\Tilt_{\IW,\Gm}$, so that $(h_\ell^s)_\Smith^* i^{!*} (\mathscr{F} \star \mathscr{T}(\zeta)) = 0$. Accordingly, since $(h_\ell^s)_\Smith^* i^{!*}$ is a triangulated functor, $(h_\ell^s)_\Smith^* i^{!*} (\mathscr{F} \star C) = 0$ and \eqref{dt2} yields a natural isomorphism
$$(h_\ell^s)_\Smith^* i^{!*} (\mathscr{F} \star \mathscr{T}(\gamma)) \cong (h_\ell^s)_\Smith^* i^{!*} (\mathscr{F} \star \mathscr{K}_\gamma).$$
Continuing from here, keeping in mind diagram \eqref{big},
\begin{align*}
    (h_\ell^s)_\Smith^* i^{!*} (\mathscr{F} \star \mathscr{K}_\gamma) & = (h_\ell^s)_\Smith^* i^{!*} (\phi_2)_* (\phi_1)^* \mathscr{F} \nonumber \\
    &= (h_\ell^s)_\Smith^* (\phi_2^\varpi)_*^\Smith (\phi_1^\varpi)_\Smith^* i^{!*} \mathscr{F} \nonumber \\
    &= w_*^\Smith \eta_\Smith^* (\phi_1^\varpi)_\Smith^* i^{!*} \mathscr{F},
\end{align*}
using smooth base change on the final line. Observe $i^{!*} \mathscr{F} = (h_\ell)_*^\Smith i_0^{!*} \mathscr{F}$, so that 
\begin{align*}
w_*^\Smith \eta_\Smith^* (\phi_1^\varpi)_\Smith^* i^{!*} \mathscr{F} &= w_*^\Smith \eta_\Smith^* (\phi_1^\varpi)_\Smith^* ((h_\ell)_*^\Smith i_0^{!*} \mathscr{F}) \nonumber \\
&= w_*^\Smith (\phi_1^\varpi \circ \eta)_\Smith^* (h_\ell)_*^\Smith (i_0^{!*} \mathscr{F}) \nonumber \\
&= w_*^\Smith \theta_*^\Smith z_\Smith^* (i_0^{!*} \mathscr{F}) \nonumber \\
&= (z^s)_*^\Smith z_\Smith^* (i_0^{!*} \mathscr{F}) \nonumber \\
&= (q^s)_*^\Smith (i_0^{!*} \mathscr{F}),
\end{align*}
where we have applied another base change and used the identities $z_* z^* \cong 1$, $z^s = q^s z$. Since the foregoing identifications are natural in $\mathscr{F}$, this establishes precisely that $$(q^s)_*^\Smith: \Smith_\IW^\lambda \to \Smith_\IW^\mu$$ 
is the translation functor corresponding to $T^s$. By considering formal properties of adjunctions, we see that $(q^s)_\Smith^*$ corresponds to $T_s$ and thus that $(q^s)_\Smith^* (q^s)_*^\Smith$ corresponds to $\theta_s$.
(This can also be proven directly by arguments analogous to those above.)

We have now constructed an action of $\mathscr{H}$ on the tilting category $\Tilt(\Rep_0(\textbf{G}))$ by wall-crossing functors. As explained in \cite[Rmk. 5.2(1)]{rw}, this induces a similar action on $\Rep_0(\textbf{G})$ and completes the proof of our main result.

\begin{thm} \label{main}
There is a monoidal right action of $\mathscr{H}$ on $\Rep_0(\textbf{G})$, such that $B_s \langle n \rangle$ acts by the wall-crossing functor $\theta_s$ for all $s \in S$, $n \in \mathbb{Z}$. That is, there exists a monoidal functor
$a: \mathscr{H} \to \End(\Rep_0(\textbf{G}))$ with $a(B_s \langle n \rangle) = \theta_s.$
\end{thm}

The following is an important corollary, already derived by Riche--Williamson, yielding the character formula for tilting modules referenced in the introduction. See \cite[\textsection 5.2]{rw} and \cite[\textsection 1.4]{rw} for proofs and additional discussion.

\begin{cor} \label{maincor}
The action of $\mathscr{H}$ on the object $T(\lambda_0) \in \Tilt(\Rep_0(\textbf{G}))$ descends to an additive functor
$$\Psi: \overline{\mathscr{H}} \to \Tilt(\Rep_0(\textbf{G}))$$
which realises $\overline{\mathscr{H}}$ as a \textit{graded enhancement} of $\Tilt(\Rep_0(\textbf{G}))$: that is, $\Psi$ induces an equivalence between the de-grading of $\overline{\mathscr{H}}$ and $\Tilt(\Rep_0(\textbf{G}))$.
\end{cor}

\begin{rmk}
We have established Theorem \ref{main} and deduced Corollary \ref{maincor} over the over the finite field $\bk$, but they hold over any extension field $\bk_0$ of $\bk$. Indeed, this can be seen via a change of base field for the category of representations and the action, since the canonical functor affords an equivalence of categories,
$$\Rep(G_{\bk_0}) \cong \bk_0 \otimes_\bk \Rep(G).$$
This remark will also apply to our more complete Theorem \ref{final} below.
\end{rmk}

\subsection{Analysis of morphisms}
It remains to examine the actions of the generating morphisms of $\mathscr{H}$. We hope to show there exist counit--unit pairs $(\varepsilon, \eta): T_s \dashv T^s$ and $(\psi, \varphi): T^s \dashv T_s$ such that \begin{equation} \label{realfirst}
a\left (
  \begin{array}{c}
    \begin{tikzpicture}[thick,scale=0.07,baseline]
      \draw (0,-5) to (0,0);
      \node at (0,0) {$\bullet$};
      \node at (0,-6.7) {\tiny $s$};
    \end{tikzpicture}
  \end{array}  \right ) = \varepsilon: \theta_s \to \id, \quad a \left (
  \begin{array}{c}
    \begin{tikzpicture}[thick,baseline,xscale=0.07,yscale=-0.07]
      \draw (0,-5) to (0,0);
      \node at (0,0) {$\bullet$};
      \node at (0,-6.7) {\tiny $s$};
    \end{tikzpicture}
  \end{array}  \right ) = \varphi: \id \to \theta_s,
\end{equation}
and
\begin{equation} \label{realsecond}
a\left ( 
  \begin{array}{c}
    \begin{tikzpicture}[thick,baseline,scale=0.07]
      \draw (-4,5) to (0,0) to (4,5);
      \draw (0,-5) to (0,0);
      \node at (0,-6.7) {\tiny $s$};
      \node at (-4,6.4) {\tiny $s$};
            \node at (4,6.4) {\tiny $s$};
    \end{tikzpicture}
  \end{array} \right ) = T_s \eta T^s: \theta_s \to \theta_s \theta_s, \quad a\left ( 
  \begin{array}{c}
    \begin{tikzpicture}[thick,baseline,scale=-0.07]
      \draw (-4,5) to (0,0) to (4,5);
      \draw (0,-5) to (0,0);
      \node at (0,-6.7) {\tiny $s$};
            \node at (-4,6.4) {\tiny $s$};
            \node at (4,6.4) {\tiny $s$};
    \end{tikzpicture}
  \end{array} \right ) = T_s \psi T^s: \theta_s \theta_s \to \theta_s.
 \end{equation}
To do this, we make an argument via total cohomology, as in \cite[\textsection 10.5]{rw}. For notational convenience, we set $R = R_{\mathbb{K}}$ from here on. Note first that by \cite[Thm. 1.3(2)]{to} we have isomorphisms
$$R \cong H_{I \rtimes \Gm}^\bullet(\pt, \mathbb{K}), \quad R^s \cong H_{\mathcal{P}^s \rtimes \Gm}^\bullet(\pt, \mathbb{K}).$$
There hence exist morphisms
\begin{equation} \label{ressc}
R \otimes_{\mathbb{K}} R \to H_{I \rtimes \Gm}^\bullet(\mathscr{Fl},\mathbb{K}), \quad \text{resp.} \quad R \otimes_{\mathbb{K}} R^s \to H_{I \rtimes \Gm}^\bullet(\Fl^s,\mathbb{K}),
\end{equation}
induced by the two actions of $I \rtimes \Gm$ on $LG$, resp. the left action of $I \rtimes \Gm$ and right action of $\mathcal{P}^s \rtimes \Gm$ on $\Fl^s$; for the latter, we use that
$$H_{I \rtimes \Gm}^\bullet(\Fl^s,\mathbb{K}) \cong H_{\mathcal{P}^s \rtimes \Gm}^\bullet((I \rtimes \Gm) \backslash (LG \rtimes \Gm),\mathbb{K}).$$
Composing the total cohomology functor $H_{I \rtimes \Gm}^\bullet(\Fl,-)$, resp. $H_{I \rtimes \Gm}^\bullet(\Fl^s,-)$, with the restrictions of scalars induced by \eqref{ressc}, we obtain functors
\begin{align*}
& \mathbb{H}: \Parity_{I \rtimes \Gm}(\Fl,\mathbb{K}) \to \text{$(R,R)$-bimod}^\mathbb{Z}, \quad \text{resp.} \\
& \mathbb{H}^s: \Parity_{I \rtimes \Gm}(\Fl^s,\mathbb{K}) \to \text{$(R,R^s)$-bimod}^\mathbb{Z}.
\end{align*}
Let us divert briefly to a general setting, for a technical observation to be used later. Consider a commutative diagram of commutative graded rings,
\begin{equation} \label{gradedsq}
\begin{tikzcd}
  B \arrow{d}{f} \arrow{r}{\beta} & B_0 \arrow{d}{f_0} \\
A \arrow{r}{\alpha} & A_0.
\end{tikzcd}
\end{equation}
\begin{lem} \label{grsq}
Each of the maps in \eqref{gradedsq} induces a functor of restriction of scalars between the associated categories of graded right modules, in such a way that the induced square of categories commutes up to natural isomorphism. Furthermore, there is a morphism of graded right $A$-modules (natural in the $B_0$-module $M$),
\begin{equation} \label{greq}
M_B \otimes_B A \to (M \otimes_{B_0} A_0)_A, \quad m \otimes a \mapsto m \otimes \alpha(a),
\end{equation}
\end{lem}

\begin{proof}
Any graded right $A$-module $M$ becomes a graded right $B$-module under the action $m \cdot b = m f(b)$; indeed, if $b \in B^i$ and $m \in M^j$ are homogeneous, then $f(b) \in A^i$ and hence $m \cdot b = m f(b)  \in M^{i+j}$. It is clear that the equality $\alpha \circ f = f_0 \circ \beta$ translates into a natural isomorphism between functors of restriction of scalars. Hence it remains only to construct the morphism \eqref{greq}. For this, recall that when its domain and target are suitably graded, the natural map $$M_B \times A \to M_B \otimes_B A$$ is graded and identifies graded $B$-balanced maps $M_B \times A \to Q$ with graded abelian group homomorphisms $M_B \otimes_B A \to Q$ for any graded abelian group $Q$; see \cite[\href{https://stacks.math.columbia.edu/tag/09LL}{Tag 09LL}]{stacks-project}. It is therefore enough to note that the assignment $(m,a) \mapsto m \otimes \alpha(a)$ is graded and $B$-balanced.
\end{proof}

\begin{prop} \label{faithful}
The functors $\mathbb{H}$ and $\mathbb{H}^s$ are faithful and fit into the two commutative squares displayed below:
\[
\begin{tikzcd}
  \Parity_{I \rtimes \Gm}(\Fl,\mathbb{K}) \arrow[d,shift right,blue] \arrow{r}{\mathbb{H}} & \text{$(R,R)$-bimod}^\mathbb{Z} \arrow[d,shift right,blue] \\
\Parity_{I \rtimes \Gm}(\Fl^s,\mathbb{K}) \arrow[u, shift right,red] \arrow{r}{\mathbb{H}^s} & \text{$(R,R^s)$-bimod}^\mathbb{Z}, \arrow[u, shift right,red]
\end{tikzcd}
\]
where the left vertical arrows are $(q^s)_*$ and $(q^s)^*$, while the right vertical arrows are restriction $(-) \otimes_R R_{R^s}$ and induction $(-) \otimes_{R^s} R_R$.
\end{prop}
\begin{proof}
For the faithfulness of $\mathbb{H}$ and $\mathbb{H}^s$, see \cite[Rmk. 3.19]{mr}. To show the square commutes up to natural isomorphism, we work with an extended diagram:
\[
\begin{tikzcd}
  \Parity_{I \rtimes \Gm}(\Fl,\mathbb{K}) \arrow[d,shift right,blue] \arrow{r}{H} & \text{$H_{I \rtimes \Gm}^\bullet(\Fl,\mathbb{K})$-mod}^{\mathbb{Z}} \arrow{r}{r} \arrow[d,shift right,blue] &  \text{$(R,R)$-bimod}^\mathbb{Z} \arrow[d,shift right,blue] \\
\Parity_{I \rtimes \Gm}(\Fl^s,\mathbb{K}) \arrow[u, shift right,red] \arrow{r}{H^s} & \text{$H_{I \rtimes \Gm}^\bullet(\Fl^s,\mathbb{K})$-mod}^{\mathbb{Z}} \arrow{r}{r^s} \arrow[u,shift right,red] & \text{$(R,R^s)$-bimod}^\mathbb{Z}. \arrow[u, shift right,red]
\end{tikzcd}
\]
Here the functors $r$ and $r^s$ are restrictions of scalars along the maps \eqref{ressc}, and the middle vertical arrows are restriction and induction. The left-most horizontal functors are known to be fully faithful, again by \cite[Rmk. 3.19]{mr}. Now, by definition, there is a natural isomorphism of graded $H^\bullet_{I \rtimes \Gm}(\Fl^s,\mathbb{K})$-modules,
$$\text{H}^\bullet_{I \rtimes \Gm}(\Fl^s,(q^s)_* \mathscr{F}) \cong \text{H}^\bullet_{I \rtimes \Gm}(\Fl,\mathscr{F}),$$
where the latter is a right module via the map $H_{I \rtimes \Gm}^\bullet(\Fl^s,\mathbb{K}) \to H_{I \rtimes \Gm}^\bullet(\Fl,\mathbb{K})$ induced by $q^s$. Notice that we have an obvious morphism 
\begin{equation} \label{hind}
H^s(\mathcal{F}) \otimes_{H_{I \rtimes \Gm}^\bullet(\Fl^s,\mathbb{K})} H_{I \rtimes \Gm}^\bullet(\Fl,\mathbb{K}) \to H((q^s)^*(\mathcal{F})).
\end{equation}
Meanwhile, Lemma \ref{grsq} proves that the square of restriction functors between graded module categories commutes; thus our claim holds for the intertwining of $(q^s)_*$ and restriction by $\mathbb{H}$ and $\mathbb{H}^s$. The same lemma provides a natural transformation
$\text{ind} \circ r^s \to r \circ \text{ind}.$
Composing on the right with $H^\bullet_{I \rtimes \Gm}(\Fl^s,-)$ and using \eqref{hind}, we obtain a morphism
\begin{equation} \label{cohnat}
\text{ind} \circ \mathbb{H}^s \to \mathbb{H} \circ (q^s)^*.
\end{equation}
We conclude by showing this is an isomorphism for objects in $\Parity_{I \rtimes \Gm}(\Fl^s,\mathbb{K})$. Note first that, for $w \in W^s$,
\begin{align*}
\mathbb{H}((q^s)^* (q_s)_* \mathscr{E}(w)) = \mathbb{H}(\mathscr{E}(w) \star \mathscr{E}(s) [ -1 ]) &= B_w \otimes_R B_s \langle -1 \rangle \\
&= (B_w \otimes_R R_{R^s}) \otimes_{R^s} R \\
&= \mathbb{H}^s((q^s)_* \mathscr{E}(w)) \otimes_{R^s} R,
\end{align*}
since $(q_s)_* \mathscr{E}(w) \cong \mathscr{E}(w) \star \mathscr{E}^s(1)$ and the total cohomology functors are monoidal. Thus \eqref{cohnat} is an isomorphism for objects of the form $(q^s)_* \mathscr{E}(w)$, $w \in W^s$. Now, every indecomposable object $\mathscr{E}^s(w)$ in the Krull--Schmidt category $\Parity_{I \rtimes \Gm}(\Fl^s,\mathbb{K})$ is a multiplicity-one factor of $(q^s)_* \mathscr{E}(w)$, so our claim follows by induction on the length of $w$ and the five lemma.
\end{proof}

Let us now consider the following morphisms in $\Parity_{I \rtimes \Gm}(\Fl,\mathbb{L})$, where $\L$ is a coefficient ring suppressed from the notation:
\begin{align*}
\Delta \left (
  \begin{array}{c}
    \begin{tikzpicture}[thick,scale=0.07,baseline]
      \draw (0,-5) to (0,0);
      \node at (0,0) {$\bullet$};
      \node at (0,-6.7) {\tiny $s$};
    \end{tikzpicture}
  \end{array}  \right ) &= u_s: \mathscr{E}(s) \to \mathscr{E}(1)[1], \quad \Delta \left (
  \begin{array}{c}
    \begin{tikzpicture}[thick,baseline,xscale=0.07,yscale=-0.07]
      \draw (0,-5) to (0,0);
      \node at (0,0) {$\bullet$};
      \node at (0,-6.7) {\tiny $s$};
    \end{tikzpicture}
  \end{array}  \right ) = \ell_s: \mathscr{E}(1) \to \mathscr{E}(s)[1]. \\
  \Delta \left ( 
  \begin{array}{c}
    \begin{tikzpicture}[thick,baseline,scale=0.07]
      \draw (-4,5) to (0,0) to (4,5);
      \draw (0,-5) to (0,0);
      \node at (0,-6.7) {\tiny $s$};
      \node at (-4,6.4) {\tiny $s$};
            \node at (4,6.4) {\tiny $s$};
    \end{tikzpicture}
  \end{array} \right ) &= b_s: \mathscr{E}(s) \to \mathscr{E}(ss)[-1], \quad \Delta \left ( 
  \begin{array}{c}
    \begin{tikzpicture}[thick,baseline,scale=-0.07]
      \draw (-4,5) to (0,0) to (4,5);
      \draw (0,-5) to (0,0);
      \node at (0,-6.7) {\tiny $s$};
            \node at (-4,6.4) {\tiny $s$};
            \node at (4,6.4) {\tiny $s$};
    \end{tikzpicture}
  \end{array} \right ) = c_s: \mathscr{E}(ss) \to \mathscr{E}(s)[-1].
\end{align*}
The first two of these give rise to natural transformations,
$$\widetilde{\varepsilon} = (-) \star u_s[-1]: (q^s)^* (q^s)_* \to \id, \quad \widetilde{\varphi} = (-) \star \ell_s: \id \to (q^s)^* (q^s)_*[2].$$

\begin{prop} \label{first}
If the coefficient ring is $\bk$, $\mathbb{O}$, or $\mathbb{K}$, then $\widetilde{\varepsilon}$ is a counit for $(q^s)^* \dashv (q^s)_*$ and $\widetilde{\varphi}$ is a unit for $(q^s)_* \dashv (q^s)^*[2]$. 
\end{prop}

\begin{proof}
For $\mathscr{F} \in \Parity_{I \rtimes \Gm}(\Fl^s,\L)$ and $\mathscr{G} \in \Parity_{I \rtimes \Gm}(\Fl,\L)$, consider maps
\begin{align*}
\widetilde{A}_{\mathscr{F},\mathscr{G}}^\L: \Hom(\mathscr{F},(q^s)_* \mathscr{G}) \to \Hom((q^s)^* \mathscr{F}, \mathscr{G}), \quad f \mapsto \widetilde{\varepsilon}_{\mathscr{G}} \circ (q^s)^*(f); \\
\widetilde{B}_{\mathscr{G},\mathscr{F}}^\L: \Hom((q^s)_* \mathscr{G}, \mathscr{F}) \to \Hom(\mathscr{G}, (q^s)^* \mathscr{F}[2]), \quad f \mapsto (q^s)^*(f)[2] \circ \widetilde{\varphi}_{\mathscr{G}}.
\end{align*}
Because we know the source and target are finite dimensional and isomorphic, it will suffice for our purposes to prove $\widetilde{A}_{\mathscr{F},\mathscr{G}}^\bk$ and $\widetilde{B}_{\mathscr{G},\mathscr{F}}^\bk$ are injective for any $\mathscr{F},\mathscr{G}$. Let us first show this is true for $\widetilde{A}_{\mathscr{F},\mathscr{G}} = \widetilde{A}_{\mathscr{F},\mathscr{G}}^\mathbb{K}$ and $\widetilde{B}_{\mathscr{G},\mathscr{F}} = \widetilde{B}_{\mathscr{G},\mathscr{F}}^\mathbb{K}$.

Suppose therefore that $\widetilde{A}_{\mathscr{F},\mathscr{G}}(f) = 0$. Then
$$0 = \mathbb{H}(\widetilde{A}_{\mathscr{F},\mathscr{G}}(f)) = \mathbb{H}(\widetilde{\varepsilon}_{\mathscr{G}}) \circ \mathbb{H}((q^s)^*(f)),$$
where $\mathbb{H}(\widetilde{\varepsilon}_{\mathscr{G}}) = \mathbb{H}(\mathscr{G}) \otimes_R \mathbb{H}(u_s[-1]) = \mathbb{H}(\mathscr{G}) \otimes_R m_s$, for $m_s: R \otimes_{R^s} R \to R$, $g \otimes h \mapsto gh$, and $$\mathbb{H}((q^s)^*(f)) = \mathbb{H}^s(f) \otimes_{R^s} R: \mathbb{H}^s(\mathscr{F}) \otimes_{R^s} R \to (\mathbb{H}(\mathscr{G}) \otimes_R R_{R^s}) \otimes_{R^s} R,$$ using \cite[\textsection 10.5]{rw} and Prop. \ref{faithful}. But then we can do explicit calculations:
$$0 = (\mathbb{H}(\mathscr{G}) \otimes_R m_s)(\mathbb{H}^s(f) \otimes_{R^s} R)(x \otimes 1) = (m_s \otimes \mathbb{H}(\mathscr{G}))(\mathbb{H}^s(f)(x) \otimes 1)$$
for any $x \in \mathbb{H}^s(\mathscr{F})$. Now, there exists $\chi \in \mathbb{H}(\mathscr{G})$ such that $\mathbb{H}^s(f)(x) = \chi \otimes 1$, and hence
$(\mathbb{H}(\mathscr{G}) \otimes_R m_s)(f(x) \otimes 1) = \chi.$
Since $x$ was arbitrary, we get $f = 0$, proving that $\widetilde{A}_{\mathscr{F},\mathscr{G}}$ is injective and therefore an isomorphism. A suitable unit for $\widetilde{\varepsilon}$ is then given by $\widetilde{A}_{\mathscr{F},(q^s)^* \mathscr{F}}^{-1}(1_{(q^s)^* \mathscr{F}})$.

On the other hand, suppose $\widetilde{B}_{\mathscr{G},\mathscr{F}}(f) = 0$, so that
$$0 = \mathbb{H}(\widetilde{B}_{\mathscr{G},\mathscr{F}}(f)) = \mathbb{H}((q^s)^*(f)) \langle 2 \rangle \circ \mathbb{H}(\widetilde{\varphi}_{\mathscr{G}}),$$
where $\mathbb{H}(\widetilde{\varphi}_{\mathscr{G}}) = \mathbb{H}(\mathscr{G}) \otimes_R \delta_s$, for $\delta_s: R \to R \otimes_{R^s} R \langle 2 \rangle$, $1 \mapsto \tfrac{1}{2}(\alpha_s \otimes 1 + 1 \otimes \alpha_s)$. Hence, for any $x \in \mathbb{H}(\mathscr{G})$,
\begin{align*}
0 = (f \otimes_{R^s} R \langle 2 \rangle)(\mathbb{H}(\mathscr{G}) \otimes_R \delta_s)(2x \otimes 1) &= (f \otimes_{R^s} R \langle 2 \rangle)(x \otimes (\alpha_s \otimes 1 + 1 \otimes \alpha_s)) \\
&= f(x \otimes \alpha_s) \otimes 1 + f(x \otimes 1) \otimes \alpha_s.
\end{align*}
However, in $\mathbb{H}^s(\mathscr{F}) \otimes_{R^s} R \langle 2 \rangle$, an equation $a \otimes 1 = b \otimes \alpha_s$ implies $a = b = 0$, so $f(x \otimes \alpha_s) = f(x \otimes \alpha_s) = 0$ and thus $f = 0$, as desired.

To transfer the conclusion from $\mathbb{K}$ to $\O$ and $\bk$, note that if $\mathscr{F} \in \Parity_{I \rtimes \Gm}(\Fl^s,\O)$ and $\mathscr{G} \in \Parity_{I \rtimes \Gm}(\Fl,\O)$, then
$\widetilde{A}_{\mathbb{K}(\mathscr{F}),\mathbb{K}(\mathscr{G})} = \mathbb{K} \otimes_{\mathbb{O}} \widetilde{A}_{\mathscr{F},\mathscr{G}}^\mathbb{O}$
by the properties of extension of scalars outlined in \textsection 3.8. Since the source of $\widetilde{A}_{\mathscr{F},\mathscr{G}}^\mathbb{O}$ is torsion free, the fact $\widetilde{A}_{\mathbb{K}(\mathscr{F}),\mathbb{K}(\mathscr{G})}$ is an isomorphism implies the same for $\widetilde{A}_{\mathscr{F},\mathscr{G}}^\mathbb{O}$. We may then extend scalars to $\bk$ to deduce that
$\widetilde{A}_{\bk(\mathscr{F}),\bk(\mathscr{G})}^\bk = \bk \otimes_{\mathbb{O}} \widetilde{A}_{\mathscr{F},\mathscr{G}}^\mathbb{O}$
is also an isomorphism. In view of \eqref{bc}, this shows that $\widetilde{A}_{\mathscr{F},\mathscr{G}}^\bk$ is an isomorphism for every $\mathscr{F},\mathscr{G}$. A similar argument applies to $\widetilde{B}_{\mathscr{G},\mathscr{F}}^\bk$.
\end{proof}

As we see from the preceding proof,
$$\widetilde{\eta}_{\mathscr{F}} = \widetilde{A}_{\mathscr{F},(q^s)^* \mathscr{F}}^{-1}(1_{(q^s)^* \mathscr{F}}), \quad \widetilde{\psi}_{\mathscr{F}} = \widetilde{B}_{(q^s)^* \mathscr{F}[2],\mathscr{F}}^{-1}(1_{(q^s)^* \mathscr{F}[2]})$$
are the unit and counit for $\widetilde{\varepsilon}$ and $\widetilde{\varphi}$, respectively. We then wish to compare $(q^s)^* \widetilde{\eta} (q^s)_*$ with $(-) \star b_s[-1]$ and $(q^s)^* \widetilde{\psi} (q^s)_*$ with $(-) \star c_s$. 

\begin{prop} \label{second}
If the coefficient ring is $\bk$, $\mathbb{O}$, or $\mathbb{K}$, then there are identifications of natural transformations,
$$(q^s)^* \widetilde{\eta} (q^s)_* = (-) \star b_s[-1], \quad (q^s)^* \widetilde{\psi} (q^s)_* = (-) \star c_s.$$
\end{prop}
\begin{proof}
We first argue over $\K$. Notice that we have natural transformations
$$\epsilon = (-) \otimes_R m_s: (-) \otimes_R B_s\langle -1 \rangle \to \id, \quad \phi = (-) \otimes_R \delta_s: \id \to (-) \otimes_R B_s \langle 1 \rangle,$$
which are such that $\epsilon \mathbb{H} = \mathbb{H} \widetilde{\varepsilon}$ and $\phi \mathbb{H} = \mathbb{H} \widetilde{\varphi}$, as well as maps of Hom spaces,
\begin{align*}
    A_{M,N}: & \Hom(M,N \otimes_R R_{R^s}) \to \Hom(M \otimes_{R^s} R,N), \quad f \mapsto \epsilon_N \circ (f \otimes_{R^s} R) \\
    B_{N,M}: & \Hom(N \otimes_R R_{R^s},M) \to \Hom(N, M \otimes_{R^s} R \langle 2 \rangle), \quad f \mapsto (f \otimes_{R^s} R\langle 2 \rangle) \circ \phi_N,
\end{align*}
which are adjunction isomorphisms for all graded $(R,R^s)$-bimodules $M$ and graded $(R,R)$-bimodules $N$. Now
\begin{align*}
A_{\mathbb{H}^s(\mathscr{F}),\mathbb{H}((q^s)^*\mathscr{F})}(\mathbb{H}^s(\widetilde{\eta}_\mathscr{F})) &= \epsilon_{\mathbb{H}((q^s)^*\mathscr{F})} \circ (\mathbb{H}^s(\widetilde{\eta}_\mathscr{F}) \otimes_{R^s} R) \\
&= \mathbb{H} ((\widetilde{\varepsilon}_{(q^s)^*\mathscr{F}}) \circ (q^s)^* \widetilde{\eta}_{\mathscr{F}}) = 1,
\end{align*}
and similarly $B_{\mathbb{H}((q^s)^*[2] \mathscr{F}),\mathbb{H}^s(\mathscr{F})}(\mathbb{H}^s(\widetilde{\psi}_{\mathscr{F}})) = 1$.
On the other hand, there are natural transformations of graded $(R,R^s)$-bimodule endofunctors, 
\begin{align*}
\zeta_M: M \to M \otimes_{R^s} R_{R^s}, & \quad m \mapsto m \otimes 1 \\
\omega_M: M \otimes_{R^s} R \langle 2 \rangle_{R^s} \to M, & \quad m \otimes r \mapsto m \partial_s(r),
\end{align*}
where $\partial_s: R \langle 2 \rangle \to R^s$ is the Demazure operator associated to $s$ \cite[\textsection 3.3]{ew}. By direct calculation, these satisfy 
\begin{align*}
A_{M, M \otimes_{R^s} R}(\zeta_M) &= \epsilon_{M \otimes_{R^s} R_{R}} \circ (\zeta_M \otimes_{R^s} R) = 1, \\
B_{M \otimes_{R^s} R,M}(\omega_M) &= (\omega_M \otimes_{R^s} R \langle 2 \rangle) \circ \phi_{M \otimes_{R^s} R} = 1.
\end{align*}
Since $A$ and $B$ are isomorphisms, the preceding calculations force
$\mathbb{H}^s(\widetilde{\eta}_{\mathscr{F}}) = \zeta_{\mathbb{H}^s(\mathscr{F})}$ and $\mathbb{H}^s(\widetilde{\psi}_{\mathscr{F}}) = \omega_{\mathbb{H}^s(\mathscr{F})}.$
Hence
\begin{align*}
    \mathbb{H}((q^s)^* \widetilde{\eta}_{(q^s)_* \mathscr{G}}) = \mathbb{H}^s(\widetilde{\eta}_{(q^s)_* \mathscr{G}}) \otimes_{R^s} R_R &= \zeta_{\mathbb{H}^s((q^s)_* \mathscr{G})} \otimes_{R^s} R_R \\ 
    &= \zeta_{\mathbb{H}(\mathscr{G}) \otimes_R R_{R^s}} \otimes_{R^s} R_R \\
    &= \mathbb{H}(\mathscr{G}) \otimes_R \mathbb{H}(b_s[-1]) = \mathbb{H}(\mathscr{G} \star b_s[-1]),
\end{align*}
using the calculation \cite[\textsection 10.5]{rw} that $\mathbb{H}(b_1): B_s \to B_{ss}\langle -1 \rangle$ is given by the formula 
$f \otimes g \mapsto f \otimes 1 \otimes g$; meanwhile,
\begin{align*}
\mathbb{H}((q^s)^* \widetilde{\psi}_{(q^s)_* \mathscr{G}}) = \mathbb{H}^s(\widetilde{\psi}_{(q^s)_* \mathscr{G}}) \otimes_{R^s} R_R 
&= \omega_{\mathbb{H}^s((q^s)_* \mathscr{G})} \otimes_{R^s} R_R \\
&= \omega_{\mathbb{H}(\mathscr{G}) \otimes_R R_{R^s}} \otimes_{R^s} R_R \\
&= \mathbb{H}(\mathscr{G}) \otimes_R \mathbb{H}(c_s) = \mathbb{H}(\mathscr{G} \star c_s),
\end{align*}
using the calculation that $\mathbb{H}(c_s): B_{ss} \to B_s \langle -1 \rangle$ is given by $f \otimes g \otimes h \mapsto f(\partial_s g) \otimes h$. Faithfulness of $\mathbb{H}$ then yields the result over $\K$.

Now, we have maps natural in $\mathscr{G} \in \Parity_{I \rtimes \Gm}(\Fl,\O)$,
\begin{align*}
\Hom(\mathscr{G} \star \mathscr{E}(s)[-1], \mathscr{G} \star \mathscr{E}(ss)[-2]) &\hookrightarrow \K \otimes_\O \Hom(\mathscr{G} \star \mathscr{E}(s)[-1], \mathscr{G} \star \mathscr{E}(ss)[-2]) \\
& \overset{\cong}{\to} \Hom(\K(\mathscr{G}) \star \mathscr{E}(s)[-1], \K(\mathscr{G}) \star \mathscr{E}(ss)[-2]).
\end{align*}
The first map is injective because its domain is torsion free. Due to the compatibilities enjoyed by extension of scalars (see again \textsection 3.8), the composite of the maps sends the natural transformation $(q^s)^* \widetilde{\eta} (q^s)_*$ over $\O$ to its counterpart over $\K$, and likewise for $(-) \star b_s[-1]$. Since these natural transformations agree over $\K$, they therefore also agree over $\O$. A further change of base to $\bk$ now proves the claim over that ring. Similar reasoning applies to $(q^s)^* \widetilde{\psi} (q^s)_*$ and $(-) \star c_s$.
\end{proof}

After successive passage to the quotient $\Parity_{\IW,\Gm}(\Fl,\bk)$, to the de-grading $\Smith_{\IW}$, and to the equivalent right $\mathscr{H}$-module category $\Tilt_0(\textbf{G})$, the relations given in Propositions \ref{first} and \ref{second} for the pairs $$(\widetilde{\varepsilon_s},\widetilde{\eta_s}): (q^s)^* \dashv (q^s)_*, \quad (\widetilde{\psi_s},\widetilde{\varphi_s}): (q^s)_* \dashv (q^s)^*[2]$$ manifest precisely as \eqref{realfirst} and \eqref{realsecond} for the pairs given by their correspondents under the various functors, $$(\varepsilon_s,\eta_s): T_s \dashv T^s, \quad (\psi_s,\varphi_s): T^s \dashv T_s.$$ This allows us to state a more complete version of Theorem \ref{main}.

\begin{thm} \label{final}
There is a monoidal right action of $\mathscr{H}$ on $\Rep_0(\textbf{G})$ such that, for all $s \in S$ and for explicit counit--unit pairs $(\varepsilon_s,\eta_s): T_s \dashv T^s$ and $(\psi_s,\varphi_s): T^s \dashv T_s$, the following properties hold:
\begin{enumerate}
    \item $B_s \langle n \rangle$ acts by the wall-crossing functor $\theta_s$ for all $n \in \mathbb{Z}$;
    \item the upper and lower dots associated to $s$ act by $\varepsilon_s$ and $\varphi_s$, respectively;
    \item the trivalent vertices associated to $s$ act by $T_s \eta_s T^s$ and $T^s \psi_s T^s$. 
\end{enumerate}
\end{thm}

\bibliographystyle{alpha}
\bibliography{biblio.bib}

\end{document}